\documentclass[a4paper,11pt]{article}
\title{\bf{Generating functions of stable pair invariants 
via wall-crossings in derived categories}}

\date{}
\author{Yukinobu Toda}

\usepackage{makeidx}

\usepackage{latexsym}
\usepackage{amscd}
\usepackage{amsmath}
\usepackage{amssymb}
\usepackage{amsthm}
\usepackage{float}
\usepackage[all]{xy}
\usepackage{graphicx}

\usepackage{array}
\usepackage{amscd}
\usepackage[all]{xy}
\usepackage{makeidx}
\usepackage{latexsym}
\DeclareFontFamily{U}{rsfs}{%
\skewchar\font127}
\DeclareFontShape{U}{rsfs}{m}{n}{%
<-6>rsfs5<6-8.5>rsfs7<8.5->rsfs10}{}
\DeclareSymbolFont{rsfs}{U}{rsfs}{m}{n}
\DeclareSymbolFontAlphabet
{\mathrsfs}{rsfs}
\DeclareRobustCommand*\rsfs{%
\@fontswitch\relax\mathrsfs}
\theoremstyle{plain}
\newtheorem{thm}{Theorem}[section]
\newtheorem{prop}[thm]{Proposition}
\newtheorem{lem}[thm]{Lemma}

\newtheorem{defi}[thm]{Definition}
\newtheorem{rmk}[thm]{Remark}
\newtheorem{cor}[thm]{Corollary}

\newtheorem{prop-defi}[thm]{Proposition-Definition}
\newtheorem{lem-defi}[thm]{Lemma-Definition}
\newtheorem{problem}[thm]{Problem}

\newtheorem{assum}[thm]{Assumption}
\newtheorem{conj}[thm]{Conjecture}
\newtheorem{exam}[thm]{Example}

\newcommand{\aA}{\mathcal{A}}
\newcommand{\bB}{\mathcal{B}}
\newcommand{\cC}{\mathcal{C}}
\newcommand{\dD}{\mathcal{D}}
\newcommand{\eE}{\mathcal{E}}

\newcommand{\gG}{\mathcal{G}}
\newcommand{\hH}{\mathcal{H}}

\newcommand{\lL}{\mathcal{L}}

\newcommand{\nN}{\mathcal{N}}
\newcommand{\oO}{\mathcal{O}}

\newcommand{\xX}{\mathcal{X}}
\newcommand{\yY}{\mathcal{Y}}

\newcommand{\lr}{\longrightarrow}

\newcommand{\Supp}{\mathop{\rm Supp}\nolimits}
\newcommand{\Hom}{\mathop{\rm Hom}\nolimits}

\newcommand{\dR}{\mathbf{R}}

\newcommand{\Var}{\mathop{\rm Var}\nolimits}

\newcommand{\Pic}{\mathop{\rm Pic}\nolimits}

\newcommand{\ch}{\mathop{\rm ch}\nolimits}
\newcommand{\rk}{\mathop{\rm rk}\nolimits}
\newcommand{\td}{\mathop{\rm td}\nolimits}
\newcommand{\Ext}{\mathop{\rm Ext}\nolimits}
\newcommand{\Spec}{\mathop{\rm Spec}\nolimits}

\newcommand{\Coh}{\mathop{\rm Coh}\nolimits}

\newcommand{\im}{\mathop{\rm im}\nolimits}

\newcommand{\cneq}{\mathrel{\raise.095ex\hbox{:}\mkern-4.2mu=}}
\newcommand{\eqcn}{\mathrel{=\mkern-4.5mu\raise.095ex\hbox{:}}}

\newcommand{\Cok}{\mathop{\rm Coker}\nolimits}

\newcommand{\Aut}{\mathop{\rm Aut}\nolimits}

\newcommand{\Stab}{\mathop{\rm Stab}\nolimits}

\newcommand{\Sch}{\mathop{\rm Sch}\nolimits}

\newcommand{\modu}{\mathop{\rm mod}\nolimits}

\newcommand{\Imm}{\mathop{\rm Im}\nolimits}
\newcommand{\Ker}{\mathop{\rm Ker}\nolimits}
\newcommand{\Ree}{\mathop{\rm Re}\nolimits}

\newcommand{\St}{\mathop{\rm St}\nolimits}

\begin{document}
\maketitle
\begin{abstract}
The notion of limit stability on Calabi-Yau 3-folds
is
introduced by the author to 
construct an approximation of Bridgeland-Douglas stability conditions 
at the large volume limit. 
It has also turned out 
that the wall-crossing phenomena of limit stable objects seem relevant to the rationality 
conjecture of the generating functions of
Pandharipande-Thomas invariants. 
In this article, we shall make it clear
 how wall-crossing formula of the 
 counting invariants of limit stable objects solves the above
 conjecture. 
\end{abstract}

\section{Introduction}
A theory of curve counting on Calabi-Yau 3-folds 
is interesting in both algebraic geometry and 
string theory. Now there are three such theories, called 
Gromov-Witten (GW) theory, Donaldson-Thomas (DT) theory, and 
Pandharipande-Thomas (PT) theory.
Conjecturally these theories are equivalent in terms of generating 
functions, however we also need a 
conjectural
rationality property of 
those functions for DT-theory and PT-theory, to
formulate that equivalence. The purpose of this article is to 
interpret the rationality conjecture for PT-theory from the 
viewpoint of wall-crossing phenomena
 in derived categories of coherent sheaves. 

\subsection{GW-DT-PT correspondences}
First of all, let us recall the conjectural GW-DT-PT correspondences 
on curve counting theories. 
Suppose that $X$ is a smooth projective Calabi-Yau 3-fold over $\mathbb{C}$, 
i.e. there is a nowhere vanishing holomorphic 3-form on $X$.
For $g \ge 0$ and $\beta \in H_2(X, \mathbb{Z})$, the 
\textit{GW-invariant} $N_{g, \beta}$ is defined by 
the integration of the virtual class, 
$$N_{g, \beta}=\int_{[\overline{M}_{g}(X, \beta)]^{\rm{vir}}}
1 \in \mathbb{Q},$$
where $\overline{M}_{g}(X, \beta)$ is the moduli 
stack of stable maps $f\colon C\to X$ with $g(C)=g$ and $f_{\ast}[C]=\beta$.
The \textit{GW-potential} is given by the following generating function, 
$$Z_{\rm{GW}}=\exp\left(\sum_{g, \beta \neq 0}
N_{g, \beta}\lambda^{2g-2}v^{\beta}\right).$$
 For $n\in \mathbb{Z}$ and 
$\beta \in H_2(X, \mathbb{Z})$,
let $I_{n}(X, \beta)$ be 
the Hilbert scheme of 1-dimensional subschemes 
$Z\subset X$ satisfying 
$$[Z]=\beta, \quad \chi(\oO_Z)=n.$$
The obstruction theory on $I_n(X, \beta)$ is obtained by 
viewing it as a moduli space of ideal sheaves, and
 the \textit{DT-invariant} $I_{n, \beta}$ 
is defined by 
$$I_{n, \beta}=\int_{[I_{n}(X, \beta)]^{\rm{vir}}}1 \in \mathbb{Z}.$$  
The generating function 
of the \textit{reduced DT-theory} is 
$$Z_{\rm{DT}}'=\sum_{n, \beta}
I_{n, \beta}q^n v^{\beta}/\sum_{n}I_{n, 0}q^n.$$
The theory of stable pairs and their counting invariants are introduced
and studied
by Pandharipande and Thomas~\cite{PT}, \cite{PT2}, \cite{PT3}
 to give a geometric interpretation 
of the reduced 
DT-theory. 
By definition, a \textit{stable pair} 
is data $(F, s)$, 
$$s\colon \oO_X \lr F, $$
where $F$ is a pure one dimensional sheaf on $X$, and 
$s$ is a morphism with a zero dimensional cokernel. 
For $\beta \in H_2(X, \mathbb{Z})$ and $n\in \mathbb{Z}$, the
 moduli 
space of stable pairs $(F, s)$ with 
$$[F]=\beta, \quad \chi(F)=n, $$
 is constructed in~\cite{PT}, denoted by 
$P_{n}(X, \beta)$. 
The obstruction theory on $P_n(X, \beta)$ is obtained by viewing 
stable pairs $(F, s)$ as two term complexes, 
\begin{align}\label{twoterm}
\cdots \lr 0 \lr \oO_X \stackrel{s}{\lr} F \lr 0 \lr \cdots.
\end{align}
The \textit{PT-invariant} $P_{n, \beta}$ is defined by  
$$P_{n, \beta}=\int_{[P_n(X, \beta)]^{\rm{vir}}}1 \in \mathbb{Z}.$$
The corresponding generating function is 
$$Z_{\rm{PT}}=\sum_{n, \beta}
P_{n, \beta}q^v v^{\beta}.$$
The functions $Z_{\rm{GW}}$, $Z_{\rm{DT}}'$ and 
$Z_{\rm{PT}}$ are conjecturally equal after suitable
variable change. In order to state this,
we need the following conjecture, called \textit{rationality conjecture}. 
\begin{conj}\label{rat}
\emph{\bf{\cite[Conjecture~2]{MNOP}, \cite[Conjecture~3.2]{PT}}}
For a fixed $\beta$, the generating series 
$$I_{\beta}(q)=
\sum_{n\in \mathbb{Z}}I_{n, \beta}q^n/\sum_{n\in \mathbb{Z}}I_{n, 0}q^n, 
\quad P_{\beta}(q)=\sum_{n\in \mathbb{Z}}P_{n, \beta}q^n,$$
are Laurent expansions of rational functions of $q$, 
invariant under $q\leftrightarrow 1/q$. 
\end{conj}
The above conjecture is solved 
for $I_{\beta}(q)$ when $X$ is a 
toric local Calabi-Yau 3-fold~\cite{MNOP}, and for $P_{\beta}(q)$
when $\beta$ is an irreducible curve class~\cite{PT3}. 
Now we can state the conjectural GW-DT-PT-correspondences. 
\begin{conj}\emph{\bf{\cite[Conjecture~3]{MNOP}, \cite[Conjecture~3.3]{PT}}}
After the variable change $q=-e^{i\lambda}$, we have 
$$Z_{\rm{GW}}=Z_{\rm{DT}}'=Z_{\rm{PT}}.$$
The variable change $q=-e^{i\lambda}$ is well-defined by 
Conjecture~\ref{rat}.
\end{conj}
Note that ideal sheaves $I\subset \oO_X$ 
are objects in $D^b(X)$, where $D^b(X)$ is the bounded 
derived category of coherent sheaves on $X$.  
We can also interpret stable pairs 
 $(F, s)$ as objects 
in $D^b(X)$ by viewing them as two term complexes (\ref{twoterm}). 
As discussed in~\cite[Secction~3]{PT}, 
the equality $Z_{\rm{DT}}'=Z_{\rm{PT}}$ should be interpreted 
as a wall-crossing formula for counting invariants in the 
category $D^b(X)$. 
The purpose of this article is to show that 
Conjecture~\ref{rat}
is also interpreted as a wall-crossing formula in $D^b(X)$, 
 using the method of 
limit stability~\cite{Tolim} together with Joyce's works~\cite{Joy1}, 
\cite{Joy2}, \cite{Joy3}, \cite{Joy4}. 

 \subsection{Limit stability}
The notion of limit stability on a Calabi-Yau 3-fold $X$ is 
introduced in~\cite{Tolim} 
to 
construct an approximation of Bridgeland-Douglas stability 
conditions~\cite{Brs1}, \cite{Dou1}, \cite{Dou2}
 on $D^b(X)$ at the large volume limit. 
It is a certain stability condition on 
the category of perverse coherent sheaves 
$$\aA^p \subset D^b(X), $$
in the sense of Bezrukavnikov~\cite{Bez} and Kashiwara~\cite{Kashi}. 
(See Definition~\ref{defiA}.)
An element $\sigma \in A(X)_{\mathbb{C}}$
determines $\sigma$-limit (semi)stable objects
in $\aA^p$, where 
$A(X)_{\mathbb{C}}$ is the complexified ample cone, 
$$A(X)_{\mathbb{C}}=\{ B+i\omega \in H^2(X, \mathbb{C}) \mid 
\omega \mbox{ is an ample class} \}.$$
It has also turned out in~\cite{Tolim}
that the objects (\ref{twoterm}) appear as 
$\sigma$-limit stable objects for some $\sigma \in A(X)_{\mathbb{C}}$, 
thus 
studying stable pairs and 
limit stable objects are closely related. 
The objects $E$ given by (\ref{twoterm}) satisfy 
\begin{align}\label{termcon}
(\ch_0(E), \ch_1(E), \ch_2(E), \ch_3(E))
=(-1, 0, \beta, n), \quad \det E=\oO_X, 
\end{align}
for some $\beta$ and $n$. 
Under the above observation, 
 we have constructed in~\cite{Tolim} the
 moduli space of $\sigma$-limit 
stable objects $E\in \aA^p$ satisfying (\ref{termcon}) 
 as an algebraic space of finite type,
 denoted by $\lL_n^{\sigma}(X, \beta)$.
Using that moduli space, 
the counting invariant of limit stable objects 
\begin{align}\label{Linv}
L_{n, \beta}(\sigma)\in \mathbb{Z}
\end{align} is also
defined in~\cite{Tolim} 
as a weighted Euler characteristic with respect to 
Behrend's constructible function~\cite{Beh},
and (\ref{Linv})
coincides with the integration of the virtual class
if $\lL_{n}^{\sigma}(X, \beta)$ is a projective variety. 
A particular choice of $\sigma$ yields an equality 
 $L_{n, \beta}(\sigma)=P_{n, \beta}$, however
 $L_{n, \beta}(\sigma)$ becomes different from $P_{n, \beta}$ if 
 we deform $\sigma$. 
As discussed in~\cite[Section~4]{Tolim}, 
a transformation formula of the invariants $L_{n, \beta}(\sigma)$
under change of $\sigma$ seems relevant to solving Conjecture~\ref{rat}
for PT-theory. 
\subsection{Main result}
In this article, we shall proceed the above idea  
further, using D.~Joyce's 
 works~\cite{Joy1}, \cite{Joy2}, \cite{Joy3}, 
\cite{Joy4} on counting invariants of 
semistable objects on abelian categories and their wall-crossing 
formulas. We will 
make it clear how such a formula for counting invariants of 
objects in $\aA^p$ 
implies Conjecture~\ref{rat} for PT-theory. 
Unfortunately we are unable to solve Conjecture~\ref{rat} 
at this moment, 
as Joyce's 
theory is applied only for the motivic invariants (e.g. Euler characteristic) 
of the moduli spaces, 
 so they do not involve
virtual classes.   
On the other hand, the invariant $P_{n, \beta}$  coincides with
the Euler characteristic of $P_{n}(X, \beta)$ (up to sign), 
$$P_{n, \beta}^{eu}\cneq e(P_{n}(X, \beta))\in \mathbb{Z}, $$
if $P_{n}(X, \beta)$
is non-singular. In general $P_{n, \beta}$ is written as 
a weighted Euler characteristic with respect to Behrend's constructible 
function~\cite{Beh}, so
 $P_{n, \beta}$ resembles $P_{n, \beta}^{eu}$ in this sense.  
So instead of solving Conjecture~\ref{rat}, 
we shall show the motivic version of Conjecture~\ref{rat}, i.e. 
the rationality of the generating series, 
$$P_{\beta}^{eu}(q)=
\sum_{n\in \mathbb{Z}}P_{n, \beta}^{eu}q^n.$$
The limit stability does not work well to combine 
Joyce's works, so we will introduce the notion of 
\textit{$\mu_{\sigma}$-limit stability} for $\sigma \in A(X)_{\mathbb{C}}$,
 which is a coarse version of $\sigma$-limit stability. Then we will introduce 
the Joyce type invariants, (cf. Definition~\ref{def:inv},
Remark~\ref{rmk:invLP}, )
$$L_{n, \beta}^{eu} \in \mathbb{Q}, \quad N_{n, \beta}^{eu} \in \mathbb{Q}. $$
Roughly speaking, 
 $L_{n, \beta}^{eu}$, 
 (resp. $N_{n, \beta}^{eu}$) is the ``Euler characteristic''
 of the moduli stack of 
 $\mu_{i\omega}$-limit semistable objects $E\in \aA^p$
with $\det E=\oO_X$, (resp. one dimensional 
$\omega$-Gieseker semistable 
sheaves $F$, ) satisfying 
$$\ch(E)=(-1, 0, \beta, n), \quad (\mbox{resp. }\ch(F)=(0, 0, \beta, n).) $$
We will consider the generating series, 
$$L_{\beta}^{eu}(q)=\sum_{n\in \mathbb{Z}}L_{n, \beta}^{eu}q^n, \quad 
N_{\beta}^{eu}(q)=\sum_{n\ge 0}nN_{n, \beta}^{eu}q^n.$$
It will turn out that $L_{\beta}^{eu}(q)$ is a polynomial of 
$q^{\pm 1}$, $N_{\beta}^{eu}(q)$ is the Laurent expansion of 
a rational function of $q$, and they are invariant under 
$q\leftrightarrow 1/q$. 
(cf.~Lemma~\ref{cor:L}, Lemma~\ref{lem:N}.)
Somewhat surprisingly, Joyce's wall-crossing 
formula yields the following equality of 
those generating functions. 
\begin{thm}\emph{\bf{[Theorem~\ref{thm:main2}]}}
We have the following equality of the generating series,
 \begin{align}\label{expan0}
 \sum_{\beta}P_{\beta}^{eu}(q)v^{\beta}=
  \left( \sum_{\beta}L_{\beta}^{eu}(q) v^{\beta} \right) \cdot
 \exp \left(\sum_{\beta}N_{\beta}^{eu}(q)v^{\beta}\right). 
 \end{align}
 \end{thm}
As a corollary, we have the following. 
 \begin{cor}\emph{\bf{[Corollary~\ref{cor:main}]}}
 \label{cor:main}
 The generating series $P_{\beta}^{eu}(q)$ is the Laurent 
 expansion of a rational function of $q$, invariant under $q\leftrightarrow 
 1/q$. 
 \end{cor} 
The series $Z_{\rm{PT}}$ also should have a decomposition such as (\ref{expan0}). 
In Problem~\ref{prob} we
 will address a certain technical problem on the 
 Ringel-Hall Lie algebra of $\aA^p$, which enables us to decompose $Z_{\rm{PT}}$ and 
solve Conjecture~\ref{rat} for PT-theory. 
As a conclusion, 
 we have obtained a conceptual understanding of 
the rationality conjecture and DT-PT correspondences 
in terms of wall-crossing phenomena in 
the derived category, and they have been reduced to showing 
a rather technical problem, namely a compatibility of 
Ringel-Hall Lie algebra structure of $\aA^p$ with 
taking virtual classes via Behrend's constructible functions. 

\subsection{Acknowledgement}
The author thanks R.~Thomas, R.~Pandharipande for valuable comments, 
and D.~Joyce for the comment on Problem~\ref{prob}. 
This work is supported by 
World Premier International 
Research Center InitiativeiWPI Initiative), MEXT, Japan.

\subsection{Convention}
All the varieties and schemes are defined over $\mathbb{C}$. 
For a variety $X$, the category of coherent sheaves on $X$ is 
denoted by $\Coh(X)$. We say $E\in \Coh(X)$ is $d$-dimensional 
if $\dim \Supp(E)=d$. 

\section{Review of Joyce's work}\label{sec:Joy}
This section is devoted to review Joyce's works~\cite{Joy1}, 
\cite{Joy2}, \cite{Joy3}, \cite{Joy4} on counting invariants
of semistable objects on abelian categories. 
We discuss in a general framework rather than working with 
the category of 
perverse coherent sheaves $\aA^p$, which we will introduce 
in the next section. 
\subsection{Setting}
We begin with a generality of (weak) stability 
conditions on abelian categories.
 Let $\aA$ be a $\mathbb{C}$-linear abelian category,
 and $K(\aA)$ its Grothendieck group.
 We put the same assumption as in~\cite{Joy4}, 
 i.e. $\Hom(E, F)$, $\Ext^1(E, F)$ for 
 any $E, F\in \aA$ are finite dimensional $\mathbb{C}$-vector
 spaces, and compositions $\Ext^{i}(E, F)\times \Ext^{j}(F, G)
 \to \Ext^{i+j}(E, G)$ for $i, j, i+j=0, 1$
 are bilinear. 
 These conditions are satisfied in several good cases, i.e. 
 $\aA=\modu A$ for a finite dimensional algebra $A$, 
 or $\aA=\Coh(X)$ for a projective variety $X$. 
 In the first case, 
the group $K(\aA)$ is finitely generated, but this is 
not true in the latter case. 
So
instead we fix a quotient space,
$$\nN(\aA)\cneq K(\aA)/\equiv, $$
for some equivalence relation $\equiv$
such that a class $[E] \in \nN(\aA)$ is non-zero for any
$0\neq E \in \aA$.
For instance if $\aA=\Coh(X)$, an 
equivalence relation $\equiv$ can be taken by
\begin{align}\label{def:equiv}
E_1 \equiv E_2 \quad \stackrel{\text{def}}{\leftrightarrow} \quad
\chi(E_1, F)=\chi(E_2, F) \quad \text{ for any }F\in \aA,
\end{align}
where $\chi(E, F)$ is defined by 
\begin{align}\label{chi}
\chi(E, F)=\sum_{i\in \mathbb{Z}}(-1)^i \dim\Ext^i(E, F).
\end{align}
Then $\nN(\aA)$ is embedded into $H^{\ast}(X, \mathbb{Q})$, and it is a
finitely generated $\mathbb{Z}$-module. 
The \textit{closed positive cone} and the
\textit{positive cone} of $\aA$ are defined by  
\begin{align*}
\overline{C}(\aA) &\cneq \im (\aA \to K(\aA) \to \nN(\aA)), \\
C(\aA) &\cneq \overline{C}(\aA)\setminus\{0\}, 
\end{align*}
respectively.
For a subcategory $\bB \subset \aA$, we shall use 
the notation $C(\bB)\cneq \im(\bB \to C(\aA)) \subset C(\aA)$, etc. 
For an object $E\in \aA$, its class is denoted by $[E] \in \overline{C}(\aA)$, 
or we omit [ ] if there is no confusion. 
Let $(T, \succeq)$ be a totally ordered set.
\begin{defi}\label{defi:stab}
\emph{
A \textit{weak stability function} is a map, 
$$Z \colon C(\aA) \lr T, $$
such that 
if $E, F, G \in C(\aA)$ satisfies 
$E=F+G$, we have either 
\begin{align*}
& Z(E)\preceq Z(F) \preceq Z(G), \quad \mbox{or} \\
& Z(E)\succeq Z(F) \succeq Z(G).
 \end{align*}
 A weak stability function is a \textit{stability function} if, 
 for $E, F, G$ as above,  
 we have either 
 \begin{align*}
 &Z(E)\prec Z(F) \prec Z(G), \quad \mbox{or} \\
 &Z(E)\succ Z(F) \succ Z(G), \quad \mbox{or} \\
 &Z(E)=Z(F)=Z(G).
 \end{align*}}
\end{defi}
Given a weak stability function, we can define 
the set of (semi)stable objects.
\begin{defi}\emph{
Let $Z\colon C(\aA) \to T$
be a weak stability function. An object $E\in \aA$ is 
called \textit{Z-(semi)stable} if for any nonzero 
subobject $F\subset E$, we have 
$$Z(F)\prec Z(E/F), \quad 
(\mbox{resp. }Z(F)\preceq Z(E/F).)$$}
\end{defi}
The notion of (weak) stability conditions 
is defined as follows. 
\begin{defi}\emph{
A (weak) stability function $Z\colon C(\aA) \to T$ is a
\textit{(weak) stability condition} if for any object $E\in \aA$, 
there is a filtration
\begin{align}\label{HN}
0=E_0 \subset E_1 \subset \cdots \subset E_n =E, \end{align}
such that each subquotient $F_i=E_i/E_{i-1}$ is
$Z$-semistable with 
$$Z(F_1)\succ Z(F_2)\succ \cdots \succ Z(F_n).$$ 
}
\end{defi}
It is easy to see that the filtration (\ref{HN}) is 
unique up to an isomorphism, if exists. The filtration (\ref{HN})
is called a \textit{Harder-Narasimhan filtration}. 
Here we give some examples. 

\begin{exam}\label{ex:stab}\emph{
 (i) For an abelian category $\aA$, let $W\colon \nN(\aA) \to \mathbb{C}$
be a group homomorphism such that for any $E\in \aA \setminus \{ 0 \}$,
we have 
$$W(E) \in \mathbb{H}\cneq \{ r\exp(i\pi\phi) \mid 0<\phi \le 1 \}.$$
For instance if $\aA=\modu A$ for a finite dimensional 
$\mathbb{C}$-algebra $A$, the positive cone $C(\aA)$ is spanned 
by finite number of simple objects $S_1, \cdots, S_n \in \aA$, 
and such $W$ is obtained by choosing the image of $[S_i] \in C(\aA)$
for $1\le i\le n$
under $W$. 
We set $(T, \succeq)=((0, 1], \ge)$, 
and 
$$Z\colon C(\aA) \ni E \longmapsto \frac{1}{\pi}\Imm \log Z(E) \in T.$$
Then $Z$ is a stability condition on $\aA$. This is 
Bridgeland's approach of stability conditions~\cite{Brs1}.}

\emph{(ii) Let $X$ be a smooth projective surface 
and set $\aA=\Coh(X)$. 
Let $\omega$ be an
ample divisor on $X$. 
For $E\in \Coh(X)$ we set 
$$\mu_{\omega}(E)
=\left\{ \begin{array}{ll} \frac{c_1(E)\cdot \omega}{\rk (E)} & 
\mbox{ if }E \mbox{ is not torsion. } \\
\infty & \mbox{ if }E\mbox{ is torsion.}
\end{array}\right.$$
Then the map $C(\aA) \ni E \mapsto \mu_{\omega}(E) 
\in \mathbb{Q}\cup \{\infty\}$
is a weak stability condition on $\aA$, but not a stability 
condition on $\aA$. 
}
\end{exam}
\begin{rmk}\emph{
Here we mention that
a theory of stability conditions 
on triangulated categories 
is developed by Bridgeland~\cite{Brs1},
motivated by M.~Douglas's $\Pi$-stability~\cite{Dou1}, \cite{Dou2}. 
For a triangulated category $\dD$, Bridgeland's stability condition
consists of $(W, \aA)$, where $\aA \subset \dD$ is the heart of 
a bounded t-structure on $\dD$, and $W$ is a group homomorphism 
$K(\aA) \to \mathbb{C}$, as in Example~\ref{ex:stab} (i). 
Especially $W$ determines a stability condition
on the abelian category $\aA$. 
He then shows that the set of ``good'' stability 
conditions form a complex manifold 
$\Stab(\dD)$. Although Bridgeland's theory is quite powerful, we 
shall study in this paper more general notion of 
(weak) stability conditions, which is used in Joyce's works.} 
\end{rmk}

\subsection{Ringel-Hall algebras}
In this subsection, we introduce the algebra $\hH(\aA)$
 associated to an abelian category $\aA$, whose 
 details are seen in~\cite{Joy2}. 
Let   
$Z\colon C(\aA) \to T$ be a weak stability function. 
At this moment, we put the following assumption. 
\begin{assum}\label{assum}

\begin{itemize}

\item $\aA$ is noetherian and $Z$-artinian. 

\item  There is an Artin stack of locally finite type
$\mathfrak{Obj}(\aA)$, which parameterizes objects $E\in \aA$. 

\item For $v\in \overline{C}(\aA)$, 
let  
 $\mathfrak{M}^v(Z) \subset \mathfrak{Obj}(\aA)$
 be the substack of $Z$-semistable objects $E\in \aA$
 with $[E]=v$. Then $\mathfrak{M}^v(Z)$ is an open substack 
 of $\mathfrak{Obj}(\aA)$, 
and it is of finite type. 
\end{itemize}
\end{assum}
Here we 
say $\aA$ is $Z$-\textit{artinian} if there is no infinite 
sequence 
$$\cdots \subset E_n \subset E_{n-1} \subset \cdots \subset E_1 \subset E_0, $$
such that 
$E_{i+1} \neq E_i$ and $Z(E_{i+1})\succeq Z(E_{i}/E_{i+1})$ for any $i$. 
The first condition of Assumption~\ref{assum}
ensures the existence of Harder-Narasimhan filtrations, 
hence $Z$ is a weak stability condition, by 
the same argument of~\cite[Theorem~2]{Ruda}.
In order to state the second assumption, we need to know 
about the notion algebraic families of objects and morphisms in $\aA$. 
This notion is obvious if $\aA=\Coh(X)$ for a variety $X$, 
but in general we need some additional extra data, which  
is given in~\cite[Assumptions~7.1, 8.1]{Joy1}. 
For the introduction of Artin stacks, one can consult~\cite{GL}. 
For instance, 
Assumption~\ref{assum} is satisfied 
when $\aA=\modu A$ for a finite dimensional $\mathbb{C}$-algebra
$A$, and $Z$ is given as in Example~\ref{ex:stab} (i).

For a variety $Y$, recall that the Grothendieck ring of 
varieties over $Y$ is defined by 
$$K_0(\Var/Y)=\bigoplus_{(X, \rho)} \mathbb{Z}[(X, \rho)]/\sim,$$
where $X$ is a variety with a morphism $\rho \colon X\to Y$, and 
equivalence relations are given by 
$$[(X, \rho)]\sim [(X^{\dag}, \rho |_{X^{\dag}})]+
[(X \setminus X ^{\dag}, 
\rho |_{X \setminus X^{\dag}})],$$
where $X^{\dag}$ is a closed subvariety of $X$.
Taking the fiber products over $Y$, there is a natural product on 
$K_0(\Var/Y)$, 
\begin{align}\label{ring}
[(X, \rho)]\cdot [(X', \rho')]=[(X\times _{Y}X', \rho \circ p)],
\end{align}
where $p$ is the projection $X\times_{Y}X' \to X$. 

In order to introduce $\hH(\aA)$, let us introduce the 
notion of Grothendieck rings of Artin stacks. 
\begin{defi}\emph{\bf{\cite{Joy5}}}
\emph{
Let 
 $\yY$ be an Artin stack of locally finite type over $\mathbb{C}$. 
Define the $\mathbb{Q}$-vector space 
$K_0(\St/\yY)$ to be 
$$K_0(\St/\yY) \cneq \bigoplus _{(\xX, \rho)}\mathbb{Q}[(\xX, \rho)]/\sim,$$
where $(\xX, \rho)$ is a pair such that $\xX$ is
an Artin $\mathbb{C}$-stack of finite type
with affine geometric stabilizers, and $\rho \colon \xX
\to \yY$
is a 1-morphism.
The relations $\sim$ are given by 
$$[(\xX, \rho)] \sim [(\xX^{\dag}, \rho |_{\xX^{\dag}})]+
[(\xX \setminus \xX ^{\dag}, 
\rho |_{\xX \setminus \xX^{\dag}})],$$
for closed substacks $\xX^{\dag}\subset \xX$. }
\end{defi}
Again taking the fiber products over $\yY$ gives a product $\cdot$
on $K_0(\St/\yY)$, 
\begin{align}\label{cdot}
[(\xX, \rho)]\cdot [(\xX', \rho')]=[(\xX\times _{\yY}\xX', \rho \circ p)],
\end{align}
where $p$ is the projection $\xX\times _{\yY}\xX' \to \xX$. 

\begin{defi}\emph{\bf{\cite{Joy2}}}
\emph{
Let $\aA$ be an abelian category satisfying the second
condition of Assumption~\ref{assum}. 
We define the $\mathbb{Q}$-vector space $\hH(\aA)$ to be 
$$\hH(\aA)\cneq K_0(\St/\mathfrak{Obj}(\aA)).$$}
\end{defi}
The vector space $\hH(\aA)$ is graded by $v\in \overline{C}(\aA)$, 
$$\hH(\aA)=\bigoplus_{v\in \overline{C}(\aA)} \hH^v(\aA), \quad 
\hH^v(\aA)\cneq K_0(\St/\mathfrak{Obj}^v(\aA)), $$
where $\mathfrak{Obj}^v(\aA)$ is the stack of objects 
$E\in \aA$ with $[E]=v$.  
  There is an associative multiplication $\ast$ 
  on $\hH(\aA)$, based on \textit{Ringel-Hall algebras}, 
  which differs from the product (\ref{cdot}). 
Let $\mathfrak{Ex}(\aA)$ be the moduli stack of 
exact sequences $0 \to E_1 \to E_2 \to E_3 \to 0$ in $\aA$. 
It is shown in~\cite[Theorem 8.2]{Joy1} that 
$\mathfrak{Ex}(\aA)$ is an Artin stack of locally 
finite type over $\mathbb{C}$. 
We have the following 1-morphisms, 
$$p_i\colon \mathfrak{Ex}(\aA) \ni (0 \to E_1 \to E_2 \to E_3 \to 0) 
\longmapsto E_i \in \mathfrak{Obj}(\aA), $$
for $i=1, 2, 3$. 
Take $f_i=[(\xX_i, \rho_i)] \in \hH(\aA)$ for $i=1, 2$.
We have the following diagram, 
$$\xymatrix{
(p_1, p_3)^{\ast}(\xX_1 \times \xX_2) 
\ar[r]^{u} \ar[d] &
\mathfrak{Ex}(\aA) \ar[r]^{p_2}\ar[d]^{(p_1, p_3)} & \mathfrak{Obj}(\aA)
 \\
 \xX_1 \times \xX_2 \ar[r]^{(\rho_1, \rho_2)} & 
\mathfrak{Obj}(\aA) \times \mathfrak{Obj}(\aA). & }$$ 
Here the left diagram is a Cartesian diagram. 
\begin{defi}\emph{
 We define the 
$\ast$-product $f_1 \ast f_2$ by 
$$f_1 \ast f_2=[((p_1, p_3)^{\ast}(\xX_1 \times \xX_2), \ p_2 \circ u)].$$}
\end{defi}
It is shown in~\cite[Theorem 5.2]{Joy2} that $\ast$ is associative and 
$\hH(\aA)$ is a $\mathbb{Q}$-algebra with identity
$[0 \hookrightarrow \mathfrak{Obj}(\aA)]$.

\begin{rmk}\label{SF}\emph{
Our algebra $\hH(\aA)$ is denoted by 
$\underline{\mathrm{SF}}(\mathfrak{Obj}(\aA))$ in Joyce's paper~\cite{Joy2}, 
and an element of $\underline{\mathrm{SF}}(\mathfrak{Obj}(\aA))$ is 
 called a \textit{stack function} on $\mathfrak{Obj}(\aA)$. 
There is another version of Ringel-Hall type algebra
discussed in~\cite{Joy2}, defined as the set of 
constructible functions on $\mathfrak{Obj}(\aA)$, 
denoted by $\mathrm{CF}(\mathfrak{Obj}(\aA))$ in~\cite{Joy2}. 
Although most of the readers might be more familiar with 
constructible functions than stack functions,  
we use the latter one  
since we want to apply~\cite[Theorem 6.12]{Joy2}
which is formulated only for stack functions.} 
\end{rmk}

\subsection{Elements $\delta^{v}(Z)$, $\epsilon^{v}(Z)$}
\label{Count}
Let $Z\colon C(\aA) \to T$ be a weak stability condition, 
satisfying Assumption~\ref{assum}. 
For an Artin
 substack $i\colon \mathfrak{M}\hookrightarrow \mathfrak{Obj}(\aA)$, 
  we write the element $[(\mathfrak{M}, i)]\in \hH(\aA)$
   as $[\mathfrak{M}\hookrightarrow 
  \mathfrak{Obj}(\aA)]$.
\begin{defi}\emph{
For $v\in C(\aA)$, we define $\delta^{v}(Z), \epsilon^{v}(Z)
 \in \hH(\aA)$
to be 
\begin{align}\notag
\delta^{v}(Z)&=[\mathfrak{M}^v(Z) \hookrightarrow 
\mathfrak{Obj}(\aA)] \in \hH^v(\aA), \\
\label{eps}
\epsilon^{v}(Z) &=
\sum_{\begin{subarray}{c}
l\ge 1, \ v_i \in C(\aA), \ v_1+\cdots v_l=v, \\
Z(v_i) = Z(v) , \ 1\le i \le l
\end{subarray}}\frac{(-1)^{l-1}}{l}\delta^{v_1}(Z)\ast \cdots
\ast \delta^{v_l}(Z) \in \hH^v(\aA).
\end{align}
}
\end{defi}
Under Assumption~\ref{assum}, 
the sum (\ref{eps}) is a finite sum,
(see~\cite[Proposition~4.9]{Joy3},) hence 
$\epsilon^{v}(Z)$ is an element of $\hH(\aA)$. 
In~\cite[Theorem~8.7]{Joy3}, Joyce shows that 
$\epsilon^{v}(Z)$ is an element of a certain
Lie subalgebra of $\hH(\aA)$, called \textit{Ringel-Hall Lie algebra,}
\begin{align}\label{Liesub}
\epsilon^{v}(Z) \in \mathfrak{G}(\aA) \subset \hH(\aA). 
\end{align}
The Lie algebra $\mathfrak{G}(\aA)$ is 
denoted by $\mathrm{SF}_{\mathrm{al}}^{\mathrm{ind}}(\mathfrak{Obj}(\aA))$
in Joyce's paper~\cite{Joy2}. 
If we work over the
Hall-type algebra $\mathrm{CF}(\mathfrak{Obj}(\aA))$, 
(see Remark~\ref{SF},) 
the corresponding Lie algebra $\mathrm{CF}^{\rm{ind}}(\mathfrak{Obj}(\aA))$
 is the set of constructible functions
on $\mathfrak{Obj}(\aA)$,  
supported on indecomposable objects.
One might expect, as an analogue for
$\hH(\aA)$, that 
an element $[(\xX, \rho)]$ is 
contained in $\mathfrak{G}(\aA)$ if the image of 
$\rho$ in $\mathfrak{Obj}(\aA)$
is supported on indecomposable objects. 
 However Joyce suggests that this definition is not the best
 analogue, and he introduces the notion of ``virtual indecomposable
  objects'', and defines
  $\mathrm{SF}_{\mathrm{al}}^{\mathrm{ind}}(\mathfrak{Obj}(\aA))$
  in~\cite[Definition~5.13]{Joy2}
  as the set of stack functions supported on
  virtual indecomposable objects. 
  We omit the precise definition of $\mathfrak{G}(\aA)$
  here, as we will not use this. 
  The Lie algebra $\mathfrak{G}(\aA)$ also has the decomposition, 
  $$\mathfrak{G}(\aA)=\bigoplus_{v\in \overline{C}(\aA)}
  \mathfrak{G}^v(\aA), \quad \mathfrak{G}^v(\aA)\cneq \hH^v(\aA)
  \cap \mathfrak{G}(\aA),$$
  and $\epsilon^{v}(Z)$ is an element of $\mathfrak{G}^v(\aA)$. 
  The conceptual meaning of the definition of $\epsilon^{v}(Z)$
  is that they are ``logarithms'' of $\delta^{v}(Z)$,
  i.e. for $t\in T$, we have formally 
  $$\sum_{v\in Z^{-1}(t)}
  \epsilon^{v}(Z)=\log \left(1+\sum_{v\in Z^{-1}(t)}\delta^{v}(Z)
  \right).$$
  Also see~\cite{BriTole} for 
   more arguments on the elements $\epsilon^{v}(Z)$.
  
  \subsection{Transformation of the elements $\delta^{v}(Z)$, 
  $\epsilon^{v}(Z)$}
   The descriptions of the variations of the 
  elements $\delta^{v}(Z)$, $\epsilon^{v}(Z)$
  under change of $Z$ are investigated in~\cite{Joy4}. 
  Let us briefly recall the main idea of~\cite[Theorem~5.2]{Joy4}
  in this subsection. 
  We first introduce the following definition. 
  \begin{defi}\label{def:dominate}
  \emph{
  Let $Z, Z'\colon C(\aA) \to T$ be weak stability conditions
  and take $v\in C(\aA)$. 
We say $Z'$ \textit{dominates} $Z$ 
with respect to $v$ 
 if for $v_1, v_2 \in C_{\le v}(\aA)$, 
$Z(v_1)\succeq Z(v_2)$ implies $Z'(v_1)\succeq Z'(v_2)$. 
Here $C_{\le v}(\aA)$ is defined by 
\begin{align}\label{CA}
C_{\le v}(\aA)=\{v' \in C(\aA) \mid \mbox{ there is }v''\in C(\aA)
\mbox{ with }v'+v''=v\}.
\end{align} }
  \end{defi}
  The first step is to show the following theorem. 
  \begin{thm}\emph{\bf{\cite[Theorem~5.11]{Joy2}}}
  \label{thm:delta}
For weak stability conditions $Z, Z' \colon C(\aA) \to T$
satisfying Assumption~\ref{assum}, 
suppose that $Z'$ dominates $Z$ with respect to $v$. 
Then we have 
\begin{align}\label{delta}
\delta^{v}(Z')=\sum_{\begin{subarray}{c}
l\ge 1, \ v_i \in C(\aA), \ v_1+\cdots +v_l=v, \\
Z(v_i)\succ Z(v_{i-1}), \ Z'(v_i)=Z'(v), \ 1\le i\le l
\end{subarray}}
\delta^{v_1}(Z)\ast \cdots \ast \delta^{v_l}(Z). 
\end{align}
The sum (\ref{delta}) may not be a finite sum, 
but it converges in the sense of~\cite[Definition 2.16]{Joy4}
\end{thm}
\begin{proof}
We just explain the idea of the proof. For the 
full proof, see~\cite[Theorem~5.11]{Joy4}. 
For a $Z'$-semistable object $E\in \aA$ with 
$[E]=v$, there is a Harder-Narasimhan filtration
with respect to $Z$, i.e. there is a unique filtration
\begin{align}\label{conv}
0=E_0 \subset E_1 \subset \cdots \subset E_l=E, 
\end{align}
such that each $F_i=E_i/E_{i-1}$ is $Z$-semistable 
with $Z(F_i)\succ Z(F_{i-1})$. Since 
$Z'$ dominates $Z$ with respect to $v$,
and each class $[F_i]\in C(\aA)$ is contained 
in $C_{\le v}(\aA)$, 
 we have $Z'(F_i) \succeq Z'(F_{i-1})$.
Hence $Z'$-semistability of $E$ implies $Z'(F_i)=Z'(F_{i-1})$.
 
Conversely for an object $E\in \aA$, suppose
that there is a filtration (\ref{conv}) 
such that $F_i=E_i/E_{i-1}$ is $Z$-semistable with 
$Z(F_i)\succ Z(F_{i-1})$ and $Z'(F_i)=Z'(F_{i-1})$
for all $i$.  
Since $Z'$ dominates $Z$ with respect to $v$, 
the object $F_i$ is also $Z'$-semistable, 
hence $E$ is $Z'$-semistable. 

As a consequence, an object $E\in \aA$ with $[E]=v$
is $Z'$-semistable if and only if there is a unique
filtration (\ref{conv}) such that 
each $F_i=E_i/E_{i-1}$ is $Z$-semistable and
$v_i=[F_i]\in C(\aA)$ for $i=1, \cdots, l$ satisfy
\begin{align}\label{satisfy}
& v_1+\cdots +v_l=v, \\
&\notag Z(v_1)\succ Z(v_2)\succ \cdots \succ Z(v_l), \\
&\notag Z'(v_1)=Z'(v_2)=\cdots =Z'(v_l).
\end{align}
This observation is expressed as 
(\ref{delta}) in terms of the algebra $\hH(\aA)$.
\end{proof}
We omit the definition of the convergence~\cite[Definition 2.16]{Joy4}
 here, as we 
will only treat the cases 
 that the relevant sums have only finitely many terms. 
The next step is to invert (\ref{delta}), and give the formula, 
\begin{align}\label{delta2}
\delta^{v}(Z)=\sum_{\begin{subarray}{c} l\ge 1, \
v_i \in C(\aA), \ 
v_1+\cdots +v_l=v, \ Z'(v_i)=Z'(v),
 \\
Z(v_1 +\cdots +v_i)\succ Z(v_{i+1}+\cdots v_{l}),  \
1\le i\le l
\end{subarray}}
\delta^{v_1}(Z')\ast \cdots \ast \delta^{v_l}(Z').
\end{align}
The proof is provided in~\cite[Theorem~5.12]{Joy4}.
The sum (\ref{delta2}) may not converge in the
sense of~\cite[Definition 2.16]{Joy4}, but
if we impose the assumption that the change from 
$Z$ to $Z'$ is \textit{locally finite}, (
we omit the definition of the local finiteness, 
see~\cite[Definition~5.1]{Joy4},)
then the sum (\ref{delta2}) converges.

Finally for two weak stability conditions $Z$, $Z'$, 
consider the following situation. 
\begin{align*}
(\spadesuit) : &\text{ there are 
weak stability conditions }
Z=Z_1, Z_2, \cdots, Z_{m}=Z', 
W_1, \cdots, W_{m-1} \text{ satisfying} \\
& \text{ Assumption~\ref{assum},
 such that }
W_i \text{ dominates }
Z_i, Z_{i+1} \text{ w.r.t. }v, \text{ and all changes from }
Z_i \text{ to} \\ 
& \ W_i, W_{i-1} \text{ are locally finite.}
\end{align*}
Then in principle one can express $\delta^{v}(Z')$ in terms of 
$\delta^v(Z)$ in the algebra $\hH(\aA)$, 
by applying the formulas (\ref{delta}), (\ref{delta2}) 
successively. 
The transformation coefficients are determined purely 
combinatory, and they are given as follows. 

\begin{defi}\emph{\bf{\cite[Definition~4.2]{Joy4}}}\emph{
Take $v_1, \cdots, v_l \in C(\aA)$ and 
weak stability conditions $Z, Z' \colon C(\aA) \to T$. 
Suppose that for each $i=1, \cdots, l-1$, we have either 
(\ref{eith1}) or (\ref{eigh2}), 
\begin{align}
\label{eith1}
Z(v_i) \preceq Z(v_{i+1}) & \mbox{ and } Z'(v_1 +\cdots + v_i) 
\succ Z'(v_{i+1}+
\cdots +v_l), \\
\label{eigh2}
Z(v_i)\succ Z(v_{i+1}) & \mbox{ and } Z'(v_1+\cdots +v_i) \preceq
 Z'(v_{i+1}+
\cdots +v_l).
\end{align}
Then define 
$S(\{v_1, \cdots, v_l\}, Z, Z')$ to be $(-1)^r$, where 
$r$ is the number of $i=1, \cdots, l-1$ satisfying (\ref{eith1}). 
Otherwise we define 
$S(\{v_1, \cdots, v_l\}, Z, Z')=0$.} 
\end{defi}
We have the following formula.  
\begin{thm}\emph{\bf{\cite[Theorem~5.2]{Joy4}}}\label{Joy1}
Under the situation $(\spadesuit)$, we have 
\begin{align}\label{S}
\delta^{v}(Z')=
\sum_{\begin{subarray}{c}l\ge 1, \ v_i \in C(\aA), \\
v_1+\cdots +v_l=v
\end{subarray}}S(\{v_1, \cdots, v_l\}, Z, Z')
\delta^{v_1}(Z)\ast \cdots \ast \delta^{v_l}(Z).
\end{align}
The sum (\ref{S}) converges in the sense of~\cite[Definition 2.16]{Joy4}.
\end{thm}
\begin{rmk}\label{relevant}\emph{
Our condition ``$\ast'$ dominates $\ast$ w.r.t. $v$''
in Definition~\ref{def:dominate} 
is weaker than Joyce's condition ``$\ast'$ dominates $\ast$''
given in~\cite[Definition~3.16]{Joy4}, and~\cite[Theorem~5.2]{Joy4}
is formulated using the latter condition.  
However if we want to know (\ref{S}) for a fixed $v\in C(\aA)$,
 it is enough to assume 
``$\ast'$ dominates $\ast$ w.r.t. $v$'' in 
$(\spadesuit)$, since all the $v_i$ in the sum
(\ref{S}) are contained in $C_{\le v}(\aA)$.} 
\end{rmk}
The relationship between 
 $\epsilon^{v}(Z')$ and $\epsilon^{v}(Z)$
 is deduced from (\ref{eps}), (\ref{S}), 
 and inverting (\ref{eps}), 
 \begin{align}\label{invert}
 \delta^{v}(Z)=\sum_{\begin{subarray}{c}
 l\ge 1, \ v_i \in C(\aA), \
 v_1+\cdots+v_l=v, \\
 Z(v_i) =Z(v), \ 1\le i\le l
 \end{subarray}}
 \frac{1}{l!}\epsilon^{v_1}(Z)\ast \cdots \ast \epsilon^{v_l}(Z). 
 \end{align}
 The proof of (\ref{invert}) is provided in~\cite[Theorem~8.2]{Joy3}.
 The transformation coefficients are given as follows.
\begin{defi}\emph{\bf{\cite[Definition~4.4]{Joy4}}}\label{defi:U}
\emph{
For $v_1, \cdots, v_l \in C(\aA)$, we
define 
$U(\{v_1, \cdots, v_l\}, Z, Z')\in \mathbb{Q}$ to be 
\begin{align}\notag
U(\{v_1, \cdots, v_l\}, Z, Z') &=\sum_{1\le m' \le m \le l}
\sum_{\begin{subarray}{c}
\text{surjective }\psi \colon \{1, \cdots, l\} \to \{1, \cdots, m\}, \
i\le j \text{ imply }\psi(i)\le \psi(j)
\\
\text{surjective } \xi \colon \{1, \cdots, m\} \to \{1, \cdots, m'\}, \
\ i\le j \text{ imply }
\xi(i)\le \xi(j), \\
\psi\text{ and }\xi \text{ satisfy }(\diamondsuit)
\end{subarray}} \\ 
& \qquad \prod_{a=1}^{m'}
S(\{ w_i \}_{i\in \xi^{-1}(a)}, Z, Z')\cdot 
 \frac{(-1)^{m'}}{m'}\cdot 
\prod_{b=1}^{m}\frac{1}{\lvert \psi^{-1}(b)\rvert!}. \label{U}
\end{align}
Here the condition $(\diamondsuit)$ is as follows. 
\begin{align*}
(\diamondsuit) :
& \text{ For }1\le i, j \le l \text{ with }\psi(i)=\psi(j),
\text{ we have }
Z(v_i)= Z(v_j), \text{ and for } 1\le i, j \le m', \text{ we have } \\
& Z'(\sum_{k\in \psi^{-1}\xi^{-1}(i)}v_k)= 
Z'(\sum_{k\in\psi^{-1}\xi^{-1}(j)}v_{k}).
\end{align*}
Also $w_i$ for $1\le i\le m$ is defined as 
$$w_i=\sum_{j\in \psi^{-1}(i)}v_j \in C(\aA).$$}
\end{defi}
\begin{thm}\emph{\bf{\cite[Theorem~5.2]{Joy4}}}
In the situation $(\spadesuit)$, the following holds.  
\begin{align}\label{TransU}
\epsilon^{v}(Z')=
\sum_{\begin{subarray}{c}l\ge 1, \ v_i \in C(\aA), \\
v_1+\cdots +v_l=v
\end{subarray}}U(\{v_1, \cdots, v_l\}, Z, Z')
\epsilon^{v_1}(Z)\ast \cdots \ast \epsilon^{v_l}(Z).
\end{align}
The sum (\ref{TransU}) converges in the sense of~\cite[Definition 2.16]{Joy4}.
\end{thm}
\begin{rmk}\emph{
It is possible to rewrite (\ref{TransU}) by a $\mathbb{Q}$-linear 
combination of multiple commutators of $\epsilon^{v_i}(Z)$ such as 
$$[[\cdots[[\epsilon^{v_1}(Z), \epsilon^{v_2}(Z)], \epsilon^{v_3}(Z)],
 \cdots], 
\epsilon^{v_l}(Z)],$$
so (\ref{TransU}) is an equality in $\mathfrak{G}(\aA)$, 
rather than in $\hH(\aA)$. 
The proof of this fact is given in~\cite[Theorem~5.4]{Joy4}.}
\end{rmk}

\subsection{Motivic invariants of stacks}\label{Motivic}
As a final step, we integrate the elements 
$\epsilon^{v}(Z)\in \mathfrak{G}^{v}(\aA)$
to give $\mathbb{Q}$-valued invariants, and 
establish the transformation formula of these invariants. Let us 
recall that a motivic invariant is a ring homomorphism
$$\Upsilon \colon K_0(\Var/\Spec\mathbb{C})\lr \Lambda,$$
where $\Lambda$ is a $\mathbb{Q}$-algebra and a
ring structure on $K_{0}(\Var/\Spec\mathbb{C})$ is given 
by (\ref{ring}).
In order to simplify the arguments,
 we only consider the special case that 
$\Lambda=\mathbb{Q}(t)$ and 
$$\Upsilon([Y])=\sum_{i}(-1)^i \dim H^i(Y, \mathbb{C})t^i, $$
where $Y$ is a smooth projective variety. 
Since $K_0(\Var/\Spec\mathbb{C})$ is generated by $[Y]$ for
smooth projective varieties $Y$, the above data 
uniquely determines $\Upsilon$. 
In this situation, there is a unique extension of $\Upsilon$, 
$$\Upsilon' \colon K_0(\mathrm{St}/\Spec\mathbb{C}) \lr \mathbb{Q}(t), $$
such that if $G$ is a special algebraic group acting on 
$Y$, we have (cf.~\cite[Theorem~4.9]{Joy5}) 
$$\Upsilon'([Y/G])=\Upsilon([Y])/\Upsilon([G]).$$
Here an algebraic group is \textit{special} if every principle
$G$-bundle is locally trivial in Zariski topology.  
In what follows, we assume 
that $\aA$ satisfies the following condition. 
\begin{align*}
(\star): &\text{ there is an anti-symmetric biadditive-paring }
\chi \colon \nN(\aA)\times \nN(\aA) \to \mathbb{Z} \\
&\text{ such that for any }E, F\in \aA, \text{ we have } \\
& \quad \chi(E, F)=\dim \Hom(E, F)- \dim \Ext^1(E, F)+ \\
& \qquad \qquad \qquad \qquad \qquad \qquad 
\dim \Ext^1(F, E)-\dim \Hom(F, E).
\end{align*}
For instance if $\aA=\Coh(X)$
for a smooth projective Calabi-Yau 3-fold $X$, 
the usual Euler pairing (\ref{chi}) 
descends to the pairing on $\nN(\aA)$, 
which satisfies $(\star)$ by Serre duality. 
Using the pairing $\chi$, we can define the following 
Lie algebra. 
\begin{defi}\emph{
For an abelian category $\aA$ satisfying $(\star)$,
we define the Lie algebra $\mathfrak{g}(\aA)$ to be the  
$\mathbb{Q}$-vector space, 
$$\mathfrak{g}(\aA)\cneq \bigoplus_{v\in \nN(\aA)}\mathbb{Q}c_v, $$
with its Lie-brackets given by $[c_{v}, c_{v'}]=\chi(v, v')c_{v+v'}$.}
\end{defi}
Let $\Pi_v$ be the composition, 
$$\Pi_v \colon \mathfrak{G}^v(\aA) \subset \hH(\aA)
 \stackrel{\pi_{\ast}}{\lr} K_0(\mathrm{St}/\Spec\mathbb{C}) 
\stackrel{\Upsilon'}{\lr} \mathbb{Q}(t), $$
where the map $\pi_{\ast}$
sends $[(\xX, \rho)]$ to $[(\xX, \pi \circ \rho)]$ 
and $\pi \colon \mathfrak{Obj}(\aA) \to \Spec\mathbb{C}$ is 
the structure morphism. 
It is shown in~\cite[Section~6.2]{Joy4} that for $\epsilon
\in \mathfrak{G}^v(\aA)$, 
the rational function 
$\Pi_v(\epsilon) \in \mathbb{Q}(t)$ has a pole at $t=1$ at most 
order one. Hence the following definition makes sense, 
\begin{align}\label{sense}
\Theta_v(\epsilon)=(t^2-1)\Pi_v(\epsilon)|_{t=1} \in \mathbb{Q}.
\end{align}
\begin{defi}\emph{We define the invariant $J^v(Z)\in \mathbb{Q}$
by 
$$J^v(Z)=\Theta_v(\epsilon^{v}(Z))\in \mathbb{Q}.$$
}
\end{defi}
\begin{rmk}\label{rmk:eu}
\emph{
If all the $Z$-semistable objects 
$E\in \aA$ with $[E]=v$
 are in fact $Z$-stable, and their moduli problem
is represented by a scheme, then  
$\epsilon^v(Z)$ is written as $[M^v(Z)/\mathbb{G}_m]$
for a scheme $M^v(Z)$. 
Here $\mathbb{G}_m$ is acting on $M^v(Z)$ trivially. 
In this case $J^v(Z)$  
equals to the Euler characteristic of $M^v(Z)$. Note that the factor 
$\Upsilon([\mathbb{G}_m])=t^2-1$
 in (\ref{sense}) is required to cancel out the contribution
of the stabilizer group $\Aut(E)\cong \mathbb{G}_m$.} 
\end{rmk}
We have the following theorem. 
\begin{thm}\emph{\bf{\cite[Theorem~6.12]{Joy2}}}\label{Lie}
The map, 
\begin{align}\label{Psi}
\Theta \colon \mathfrak{G}(\aA)=\bigoplus_{v\in \overline{C}(\aA)}
 \mathfrak{G}^{v}(\aA)
 \ni \{ \epsilon^v \}_{v} \longmapsto 
\sum_{v\in \overline{C}(\aA)} \Theta_v(\epsilon^v)c_v \in \mathfrak{g}(\aA), 
\end{align}
is a Lie algebra homomorphism.
\end{thm} 
Since (\ref{TransU}) is a relationship in the Lie algebra $\mathfrak{G}(\aA)$, 
we can obtain the relationship between $J^v(Z')$ and $J^v(Z)$ by 
applying $\Theta$. The result is as follows. 
\begin{thm}\emph{\bf{\cite[Theorem~6.28, Equation (130)]{Joy4}}}
\label{prop:trans}
In the situation of $(\spadesuit)$, assume that 
there are only finitely many terms in (\ref{TransU}). 
Applying $\Theta$ to (\ref{TransU}) yields the formula, 
\begin{align}
\notag
J^v(Z')=&\sum_{\begin{subarray}{c}l\ge 1, \ v_i \in C(\aA), \\
v_1+\cdots +v_l=v\end{subarray}}
\sum _{\begin{subarray}{c}
\Gamma \text{ \rm{is a connected, simply connected}} \\
\text{\rm{graph with vertex} }\{1, \cdots, l\}, \
\stackrel{i}{\bullet}\to \stackrel{j}{\bullet}\text{ \rm{implies} }
i< j
\end{subarray}} \\
\label{Trans2}
&\qquad \frac{1}{2^{l-1}}U(\{v_1, \cdots, v_l\}, Z, Z')
\prod_{\stackrel{i}{\bullet} \to \stackrel{j}{\bullet}\text{ \rm{in} }\Gamma}
\chi(v_i, v_j)\prod_{i=1}^{l}J^{v_i}(Z).
\end{align}
\end{thm}

\subsection{Generalization to quasi-abelian categories}\label{sub:quasi}
For our purpose it is useful to give a slight generalization of
Theorem~\ref{prop:trans}, especially we want to relax 
Assumption~\ref{assum}, as our abelian category $\aA^p$
which will be introduced in the next section does not 
satisfy that assumption.  
Let $\aA$ be an abelian category and 
$Z\colon C(\aA) \to T$ a weak stability condition. 
Here we do not assume Assumption~\ref{assum}. 
For $t\in T$, we set 
\begin{align*}
\aA_{Z\succeq t}&=\langle E \mid E\mbox{ is }Z\mbox{-semistable with }
Z(E) \succeq t \rangle, \\
\aA_{Z\prec t}&=\langle E \mid E\mbox{ is }Z\mbox{-semistable with }
Z(E) \prec t \rangle.
\end{align*}
Here for a set of objects $S$ in $\aA$, we denote by 
$\langle S \rangle \subset \aA$ the smallest extension closed 
subcategory of $\aA$ which contains $S$. 
Equivalently, an object $E\in \aA$ is contained in $\aA_{Z\succeq t}$ 
(resp. $\aA_{Z\prec t}$)
if and only if any $Z$-semistable factor $F$ of $E$ 
satisfies $Z(F)\succeq t$, (resp. $Z(F)\prec t$.)
It can be shown that $\aA_{Z\succeq t}$, 
$\aA_{Z\prec t}$ are quasi-abelian categories. 
See~\cite[Section~4]{Brs1}
for the detail on quasi-abelian categories. 
\begin{defi}\label{strict}\emph{
For objects $E, F\in \aA_{Z\prec t}$ and a morphism  
$f\colon E\to F$, it is called a \textit{strict monomorphism}
if $f$ is injective in $\aA$ and $\Cok(f) \in \aA_{Z\prec t}$. 
Similarly $f$ is called a \textit{strict epimorphism} if 
$f$ is surjective in $\aA$ and $\Ker(f) \in \aA_{Z\prec t}$.} 
\end{defi}
For $v\in C(\aA_{Z\prec t})$, we set 
\begin{align}\label{Cle}
C_{\le v}(\aA_{Z\prec t})=\{ v'\in C(\aA_{Z\prec t}) \mid 
\mbox{ there is }v'' \in C(\aA_{Z\prec t}) \mbox{ with }
v'+v''=v\}.
\end{align}
For $Z, t, v$ as above, we
put the following 
assumption instead of Assumption~\ref{assum}. 
\begin{assum}\label{assum2}
\begin{itemize}
\item The category $\aA_{Z\prec t}$ is noetherian and artinian with respect
to strict monomorphisms. 
\item There is an Artin stack of locally finite 
type $\mathfrak{Obj}(\aA)$, which parameterizes objects $E\in \aA$. 
\item For any $v' \in C_{\le v}(\aA_{Z\prec t})$, the substack 
$\mathfrak{M}^{v'}(Z)\subset \mathfrak{Obj}(\aA)$ is an 
open substack, and it is of finite type. 
\end{itemize}
\end{assum}
We also modify the dominant conditions. 
\begin{defi}\emph{
Let $Z, Z' \colon C(\aA) \to T$ be  
weak stability conditions. For 
$t\in T$ and 
$v\in C(\aA_{Z\prec t})$, we say 
$Z'$ \textit{dominates} $Z$ \textit{with respect to} $(v, t)$ if 
the following holds.}
\begin{itemize}
\item \emph{We have $\aA_{Z\succeq t}=\aA_{Z'\succeq t}$ and 
$\aA_{Z\prec t}=\aA_{Z'\prec t}$.} 
\item \emph{For $v_1, v_2 \in C_{\le v}(\aA_{Z\prec t})$, 
if $Z(v_1)\succeq Z(v_2)$
then $Z'(v_1)\succeq Z'(v_2)$. }
\end{itemize}
\end{defi}
For two weak stability conditions $Z$, $Z'$, we consider the 
following situation. 
\begin{align*}
(\spadesuit') : &\text{ there are 
weak stability conditions }
Z=Z_1, Z_2, \cdots, Z_{m}=Z', 
W_1, \cdots, W_{m-1} 
\text{ such that}\\
& \ W_i \text{ dominates }
Z_i, Z_{i+1} \text{ w.r.t. }(v, t), \text{ all changes from } 
 Z_i \text{ to } W_i, W_{i-1} \text{ are locally finite,}\\
 & \text{ and } Z_i, W_i \text{ satisfy Assumption~\ref{assum2} for }
 (v, t).
\end{align*}
We have the following generalization of Theorem~\ref{prop:trans}.
\begin{thm}\label{thm:gen}
For weak stability 
conditions $Z$, $Z'$ on $\aA$, $t\in T$ and $v\in C(\aA_{Z\prec t})$, 
suppose that the condition $(\spadesuit')$
holds. Then the equation (\ref{TransU}) with each
$v_i \in C(\aA_{Z\prec t})$ holds. If there are only 
finitely many terms in (\ref{TransU}) with $v_i \in C(\aA_{Z\prec t})$, 
then (\ref{Trans2}) holds, with 
each $v_i \in C(\aA_{Z\prec t})$. 
\end{thm}
\begin{proof}
First suppose that $Z'$ dominates $Z$ with respect to $(v, t)$, 
and check that (\ref{delta}) holds with each
$v_i \in C(\aA_{Z\prec t})$. Let $E\in \aA$ be 
$Z'$-semistable with $[E]=v$. As in the proof of Theorem~\ref{thm:delta}, 
we have a unique filtration (\ref{conv}). 
Let $F_i=E_i/E_{i-1}$, $v_i=[F_i]\in C(\aA)$. 
Since $v\in C(\aA_{Z\prec t})=C(\aA_{Z'\prec t})$, we have 
$E\in \aA_{Z'\prec t}=\aA_{Z\prec t}$. 
Hence we have 
$Z(v_i) \prec t$, and $v_i \in C(\aA_{Z\prec t})$ follows. 

Conversely given a filtration (\ref{conv}), suppose that 
each $F_i$ is $Z$-semistable with $v_i \in C(\aA_{Z\prec t})$. 
Then $F_i \in \aA_{Z\prec t}=\aA_{Z'\prec t}$, hence 
$F_i$ is also $Z'$-semistable as $Z'$ dominates $Z$ w.r.t. 
$(v, t)$.

 As a summary, an object $E\in \aA_{Z\prec t}$ is 
$Z$-semistable if and only if there is a filtration (\ref{conv}), 
satisfying (\ref{satisfy}) with each $v_i \in C(\aA_{Z\prec t})$. 
Then the same proof of~\cite[Theorem~5.11]{Joy4}
works and gives the formula (\ref{delta}) with each 
$v_i \in C(\aA_{Z\prec t})$. 
Note that to state the formula (\ref{delta}), it
is enough to assume that $\mathfrak{M}^{v'}(Z)\subset \mathfrak{Obj}(\aA)$
is open and of finite type for any $v' \in C_{\le v}(\aA_{Z\prec t})$. 

By the same idea, we can also show the formulas
 (\ref{delta2}), (\ref{S}), (\ref{invert}), (\ref{TransU})
 hold with each $v_i \in C(\aA_{Z\prec t})$. 
 We leave the readers to follow Joyce's work
 and that the same proofs are applied in this case. 
\end{proof}

\section{Limit stability and $\mu$-limit stability}\label{sec:lim}
In this section, we recall the notion of limit 
stability on a Calabi-Yau 3-fold $X$
introduced in~\cite{Tolim}, and
also introduce the notion of $\mu$-limit stability. 
Below we always assume that $X$ is a projective 
complex 3-fold with a trivial canonical class, i.e. 
$X$ is a Calabi-Yau 3-fold. 
We denote by $D^b(X)$ the bounded derived category 
of $\Coh(X)$, and 
 $K(X)$ the Grothendieck group of 
$\Coh(X)$. 
The \textit{numerical Grothendieck group}
of $\Coh(X)$ is given by
 $$\nN(X)=K(X)/\equiv,$$ 
 where the numerical equivalence relation $\equiv$
is given by (\ref{def:equiv}).  
Note that if $\aA \subset D^b(X)$ is the heart 
of a bounded t-structure on $D^b(X)$, 
then the group $\nN(\aA)=K(\aA)/\equiv$, where 
$\equiv$ is (\ref{def:equiv}) as above,
coincides with $\nN(X)$. 
So we always regard $C(\aA)$ as a subset of $\nN(X)$. 

Let us fix notation of the numerical classes of curves on
$X$. An element $\beta \in H^4(X, \mathbb{Z})$ is 
called an \textit{effective class} if there is a one dimensional 
subscheme $C\subset X$ such that 
$\beta$ is the Poincar\'e dual of 
the fundamental cycle of $C$. 
We set $C(X)$, $\overline{C}(X)$ as 
\begin{align*}
C(X) &\cneq \{ \beta \in H^4(X, \mathbb{Z}) \mid \beta
\mbox{ is an effective class }\}, \\ 
\overline{C}(X) &\cneq C(X) \cup \{0\}.
\end{align*}

\subsection{Definition of limit stability}
The limit stability
introduced in~\cite{Tolim} is a stability condition on
the category of perverse coherent sheaves $\mathcal{A}^p$
in the sense of 
Bezrukavnikov~\cite{Bez} and Kashiwara~\cite{Kashi}, 
and it is also one of the polynomial stability conditions 
introduced by Bayer~\cite{Bay} independently. 
In order to introduce $\mathcal{A}^p$, let us define 
the subcategories $(\Coh_{\le 1}(X), \Coh_{\ge 2}(X))$
of $\Coh(X)$, as follows. 
\begin{defi}\emph{We define the pair of subcategories
$(\Coh_{\le 1}(X), \Coh_{\ge 2}(X))$ to be  
\begin{align*}
\Coh_{\le 1}(X) &\cneq \{ E\in \Coh(X) \mid \dim \Supp(E) \le 1\}, \\
\Coh_{\ge 2}(X) &\cneq \{E\in \Coh(X) \mid \Hom(\Coh_{\le 1}(X), E)=0\}.
\end{align*}}
\end{defi}
The category $\mathcal{A}^p$ is defined as follows. 
\begin{defi}\label{defiA}
\emph{
We define the subcategory $\aA^p \subset D^b(X)$ to be  
$$\aA^p \cneq \left\{ E\in D^b(X) : \begin{array}{ll}
 \hH^{-1}(E) \in \Coh_{\ge 2}(X), 
\hH^{0}(E) \in \Coh_{\le 1}(X), \\
\mbox{ and } \hH^{i}(E)=0 \mbox{ for }i\neq -1, 0.
\end{array}\right\}.$$}
\end{defi}
It is easy to see that 
$(\Coh_{\le 1}(X), \Coh_{\ge 2}(X))$
determines a torsion theory on $\Coh(X)$, and  
$\aA^p$ is the corresponding tilting. 
(cf.~\cite[Definition~2.9, Lemma~2.10]{Tolim}.)
Therefore $\aA^p$ is the heart of 
a bounded t-structure on $D^b(X)$, 
in particular $\aA^p$ is an abelian category. 

Next recall that the \textit{complexified ample cone} is defined by 
$$A(X)_{\mathbb{C}}=\{ B+i\omega \in H^2(X, \mathbb{C}) \mid 
\omega \mbox{ is an ample class}\}.$$
Given $\sigma=B+i\omega \in A(X)_{\mathbb{C}}$, one can 
define the map $Z_{\sigma}\colon K(X) \to \mathbb{C}$, 
$$Z_{\sigma}\colon K(X) \ni E \longmapsto
 -\int e^{-(B+i\omega)}\ch(E)\sqrt{\td_X} \in \mathbb{C}.$$
 Explicitly we have 
 \begin{align}\label{explicit}
 Z_{\sigma}(E)=
 \left( -v^B_{3}(E)+\frac{1}{2}\omega ^2 v^B_1(E) \right) 
+ \left( \omega v_2^{B}(E) -\frac{1}{6}\omega^3 v^B_0(E)\right)i,
\end{align}
where $v_i^B \in H^{2i}(X, \mathbb{R})$ for $0\le i\le 3$ are given by
$$e^{-B}\ch(E)\sqrt{\td_X}=(v_0^{B}(E), v_1^{B}(E), 
v_2^{B}(E), v_3^{B}(E))\in H^{\rm{even}}(X, \mathbb{R}).$$ 
For $\sigma_{m}=B+im\omega$ for $m \in \mathbb{R}$, one 
can show the following: for each non-zero object $E\in \aA^p$, 
one has
\begin{align}\label{eq:phase}
Z_{\sigma_m}(E) \in \left\{
r\exp(i\pi \phi) : r>0, \ \frac{1}{4}<\phi<\frac{5}{4}\right\},
\end{align}
for $m\gg 0$. (See~\cite[Lemma~2.20]{Tolim}.)
Hence the \textit{phase} of $E$ is well-defined for 
$m\gg 0$ as follows, 
$$\phi_{\sigma_m}(E) =\frac{1}{\pi}\Imm \log Z_{\sigma_{m}}(E)
\in \left( \frac{1}{4}, \frac{5}{4}\right).$$
\begin{defi}\emph{\bf{\cite[Definition~2.21]{Joy4}}}
\emph{
An object $E\in \aA^p$ is $\sigma$-\textit{limit (semi)stable} 
if for any non-zero subobject $F\subset E$ in $\aA^p$, we have 
\begin{align}\label{def:lim}
\phi_{\sigma_m}(F) < \phi_{\sigma_m}(E), \quad (\mbox{resp. }
\phi_{\sigma_m}(F) \le \phi_{\sigma_m}(E), ) 
\end{align}
for $m\gg 0$.
}
\end{defi}
Let us interpret the above 
stability in terms of Definition~\ref{defi:stab}. 
Let $T$ be the one valuable function field 
$\mathbb{R}(m)$. 
We define the total order on $\mathbb{R}(m)$ to be 
$$f(m)\succeq g(m) \quad \stackrel{\rm{def}}{\leftrightarrow} 
\quad 
f(m)  \ge g(m) \text{ for } m\gg 0.$$
Note that we have
\begin{align}\label{poly}
\Imm e^{-\pi i/4}Z_{\sigma_m}(E) >0, 
\end{align}
for $m\gg 0$ since (\ref{eq:phase}) holds. 
In particular (\ref{poly}) is non-zero as a polynomial 
of $m$, thus the following map is well-defined. 
$$Z_{\sigma}^{T}\colon 
C(\aA^p) \ni E \longmapsto  
-\frac{\Ree e^{-\pi i/4}Z_{\sigma_m}(E)}{\Imm e^{-\pi i/4}Z_{\sigma_m}(E)}
\in T.$$
\begin{lem}\label{circ}
The map $Z_{\sigma}^{T}$ is a stability condition
on $\aA^p$, and
an object
$E\in \aA^p$ is $Z_{\sigma}^{T}$-(semi)stable if and 
only if $E$ is $\sigma$-limit (semi)stable. 
\end{lem}
\begin{proof}
Since (\ref{poly}) holds, it is obvious that 
for $v_1, v_2 \in C(\aA_{p})$, 
the inequality $\phi_{\sigma_m}(v_1)\le  \phi_{\sigma_m}(v_2)$ holds 
for $m\gg 0$
if and only if $Z_{\sigma}^{T}(v_1) \preceq Z_{\sigma}^{T}(v_2)$
holds in $T$. 
Hence $Z_{\sigma}^{T}$ is a stability function, and 
the latter statement also follows. 
The existence of Harder-Narasimhan filtrations for limit stability is proved in~\cite[Theorem~2.28]{Tolim}, 
hence $Z_{\sigma}^{T}$ is a stability condition.  
\end{proof}

\subsection{$\mu$-limit stability}
In this subsection, we introduce a weak
stability condition on $\aA^p$, which we call
\textit{$\mu$-limit stability}. 
Let us introduce the following notation. 
\begin{defi}\emph{
Let $f=\sum_{i=0}^{d} a_i(\sigma) m^i$
be a polynomial such that each coefficient $a_i(\sigma)$ is a
$\mathbb{R}$-valued function on $A(X)_{\mathbb{C}}$, and
 $a_d(\sigma) \not \equiv 0$. We define 
 $f^{\dag}=a_d(\sigma) m^d$. }
\end{defi}
By the formula (\ref{explicit}), 
$\Ree Z_{\sigma _m}(E)$ and $\Imm Z_{\sigma_m}(E)$ 
are written as polynomials of $m$ whose coefficients
are $\mathbb{R}$-valued functions on $A(X)_{\mathbb{C}}$. 
Thus the following makes sense,
$$Z_{\sigma_m}^{\dag}(E) \cneq 
(\Ree Z_{\sigma_m}(E))^{\dag}+(\Imm Z_{\sigma_m}(E))^{\dag}i.$$
The same argument of~\cite[Lemma~2.20]{Tolim}
shows that 
$$Z_{\sigma_m}^{\dag}(E) \in \left\{ r\exp(i\pi \phi) : r>0, \
\frac{1}{4}< \phi < \frac{5}{4} \right\}, $$
for $m\gg 0$. Hence as before we can define the following map, 
$$Z_{\mu_{\sigma}}\colon C(\aA^p) \ni E \longmapsto 
-\frac{\Ree e^{-\pi i/4} 
Z_{\sigma_m}^{\dag}(E)}{\Imm e^{-\pi i/4} Z_{\sigma_m}^{\dag}(E)}
\in T.$$ 
We have the following lemma. 
\begin{lem}
The map $Z_{\mu_{\sigma}}$ is a weak stability function on 
$\aA^p$. 
\end{lem}
\begin{proof}
Since the operation $f \mapsto f^{\dag}$ is just taking the 
initial term of the polynomials, we have the implication
\begin{align}\label{imply}
Z_{\sigma}^{T}(E)\succ Z_{\sigma}^{T}(F) \quad 
\Rightarrow \quad 
 Z_{\mu_{\sigma}}(E) \succeq Z_{\mu_{\sigma}}(F),
 \end{align}
 for $E, F\in C(\aA^p)$. Hence 
$Z_{\mu_{\sigma}}$ is a weak stability function. 
\end{proof}
It is easy to see that for $0\neq E \in \aA^p$, one has 
 $Z_{\mu_{\sigma}}(E) \to 1$ or $-1$ for $m\to \infty$.
To see that $Z_{\mu_{\sigma}}$ is a weak stability
condition, we introduce a pair of subcategories in $\aA^p$. 
(cf.~\cite[subsection~2.3]{Tolim}.)
\begin{defi}\emph{
We define $(\aA_{1}^p, \aA_{1/2}^p)$ to be 
\begin{align*}
\aA_{1}^p &=\{ E\in \aA^p \mid \dim \Supp \hH^0(E)=0 \mbox{ and }
\hH^{-1}(E) \mbox{ is a torsion sheaf }\}, \\
\aA_{1/2}^p &=\{ E\in \aA^p \mid \Hom(F, E)=0 \mbox{ for any }
F\in \aA_{1}^p\}.
\end{align*}}
\end{defi}
It is shown in~\cite[Lemma~2.16]{Tolim}
that 
$(\aA_{1}^{p}, \aA_{1/2}^p)$ determines a torsion theory on $\aA^p$.
We shall use the same notation of ``strict monomorphisms'', 
``strict epimorphisms'' in $\aA_{i}^p$ as in Definition~\ref{strict}, 
i.e. we replace $\aA_{Z\prec t}$ by $\aA_{i}^p$ to define them. 
We have the following. 
\begin{lem}\label{same}
An object $E\in \aA^p$ is $Z_{\mu_{\sigma}}$-semistable 
with $Z_{\mu_{\sigma}}(E) \to -1$ (resp. $1$)
for $m\to \infty$
if and only if $E\in \aA^p_{1/2}$, (resp. $\aA_{1}^{p}$,)
 and for any strict monomorphism  
$0\neq F\hookrightarrow E$ in $\aA_{1/2}^p$,
 (resp. $\aA_{1}^{p}$,)
  one has 
$Z_{\mu_{\sigma}}(F)\preceq Z_{\mu_{\sigma}}(E/F)$. 
\end{lem}
\begin{proof}
The proof is same as in~\cite[Lemma~2.26]{Tolim}
for the limit stability,  
by noting that 
$$Z_{\mu_{\sigma}}(E) \to -1, \quad (\mbox{resp. }Z_{\mu_{\sigma}}(E) \to 1,)$$
for $E\in \aA_{1/2}^p$, (resp. $E\in \aA_{1}^{p}$,)
and $m \to \infty$. 
\end{proof}
We also have the following. 
\begin{lem}
The weak stability function $Z_{\mu_{\sigma}}$ is 
a weak stability condition. 
\end{lem}
\begin{proof}
The existence of Harder-Narasimhan filtrations follows from 
the same argument for the limit stability. 
The proof of~\cite[Theorem~2.28]{Tolim} also works in this case, by
noting Lemma~\ref{same}. 
\end{proof}
We say $E\in \aA^p$ is $\mu_{\sigma}$-limit 
(semi)stable if it is (semi)stable 
in $Z_{\mu_{\sigma}}$-weak stability. 
To explain this notation, let us recall that 
the (usual) $\mu$-stability is defined by  
cutting off the lower degree terms of the reduced 
Hilbert polynomials. 
In this sense, our $\mu_{\sigma}$-limit stability 
resembles to $\mu$-stability, as we also cut 
off the lower degree terms of the polynomial $Z_{\sigma_m}(\ast)$. 
By (\ref{imply}), we have the following implications, 
$$\mu_{\sigma}\mbox{-limit stable } \Rightarrow \sigma\mbox{-limit stable }
\Rightarrow \sigma\mbox{-limit semistable } \Rightarrow 
\mu_{\sigma}\mbox{-limit semistable }.$$
\begin{rmk}\emph{
For $0\neq E \in \aA^p$, 
let 
$\phi_{\sigma_m}^{\dag}(E)$ be
$$\phi_{\sigma_m}^{\dag}(E)=\frac{1}{\pi}\Imm \log Z_{\sigma_m}^{\dag}(E)
\in \left( \frac{1}{4}, \frac{5}{4} \right).$$
Then obviously an object $E\in \aA^p$ is $\mu_{\sigma}$-limit 
(semi)stable if and only if for any non-zero subobject 
$F\subset E$, one has 
$$\phi_{\sigma_m}^{\dag}(F) < \phi_{\sigma_m}^{\dag}(E/F), \quad 
(\mbox{resp. }\phi_{\sigma_m}^{\dag}(F) \le  \phi_{\sigma_m}^{\dag}(E/F),)$$
for $m\gg 0$. 
}
\end{rmk}

\begin{rmk}\label{rmk:twisted}
\emph{
Let us take $F\in \Coh_{\le 1}(X)$. 
In this case we have $Z_{\sigma_m}^{\dag}(F)=Z_{\sigma_m}(F)$, 
and $\Coh_{\le 1}(X)\subset \aA^p$ is closed under 
taking subobjects and quotients. So
 $F$ is $\mu_{\sigma}$-limit (semi)stable 
if and only it is $\sigma$-limit (semi)stable.  
On the other hand for 
$\sigma=B+i\omega \in A(X)_{\mathbb{C}}$, let $\mu_{\sigma}(F) \in \mathbb{Q}$
be 
\begin{align}\label{mu}
\mu_{\sigma}(F)=\frac{\ch_3(F)-B\ch_2(F)}{\omega \ch_2(F)} \in \mathbb{Q}.
\end{align}
As in~\cite[Example~2.24 (ii)]{Tolim}, the object
$F$ is $\sigma$-limit (semi)stable if and only if for any 
subsheaf $0\neq F'\subset F$ we have 
$\mu_{\sigma}(F') < \mu_{\sigma}(F)$, (resp. 
$\mu_{\sigma}(F') \le \mu_{\sigma}(F)$,) 
i.e. $F$ is a $(B, \omega)$-twisted (semi)stable sheaf. 
If $B=k\omega$ for $k\in \mathbb{R}$, then 
$F$ is $\sigma$-limit (semi)stable if and only if 
$F$ is a $\omega$-Gieseker (semi)stable sheaf, and this notion 
does not depend on $k$. 
}
\end{rmk}
\begin{rmk}\label{rmk:ob}
\emph{
By Lemma~\ref{same}, it is obvious that  
\begin{align}\label{obv}
\aA^{p}_{Z_{\mu_{\sigma}}\prec 0}=\aA_{1/2}^{p}, \quad 
\aA^{p}_{Z_{\mu_{\sigma}}\succeq 0}=\aA_{1}^{p},
\end{align}
in the notation of subsection~\ref{sub:quasi}. 
Since $\aA_{1/2}^p$ and $\aA_{1}^p$ are of finite length with 
respect to strict monomorphisms and strict 
epimorphisms, (cf.~\cite[Lemma~2.19]{Tolim},)
the categories 
$\aA^{p}_{Z_{\mu_{\sigma}}\prec 0}$
and $\aA^{p}_{Z_{\mu_{\sigma}}\succeq 0}$ also have such 
properties. 
}
\end{rmk}

\subsection{Characterization of $\mu$-limit semistable objects}
Take $v\in C(\aA^p)$ satisfying 
 \begin{align}\label{num}
 (\ch_0(v), \ch_1(v), \ch_2(v), \ch_3(v))
 =(-1, 0, \beta, n),
 \end{align}
 for some $\beta \in H^4(X, \mathbb{Q})$ and 
 $n\in H^6(X, \mathbb{Q}) \cong \mathbb{Q}$. 
In this subsection we give a
 characterization of $\mu$-limit semistable objects
 $E\in \aA^p$ of numerical type $v$. 
 i.e. $[E]=v \in C(\aA^p)$. 
 Note that such objects satisfy 
 $Z_{\mu_{\sigma}}(E) \to -1$ for $m\to \infty$, hence we have 
 $$E\in \aA_{1/2}^p, \quad v\in C(\aA_{1/2}^p).$$
 Also if such $E$ exists, 
  the classes $\beta$, $n$ are contained in 
 $\overline{C}(X)$, $H^6(X, \mathbb{Z})\cong \mathbb{Z}$ respectively,
 by~\cite[Remark~3.3]{Tolim}.
 We have the following proposition, whose 
 corresponding result for limit stability is
  seen in~\cite[Section~3]{Tolim}.

\begin{prop}\label{prop:char}
Take $\sigma=B+i\omega \in A(X)_{\mathbb{C}}$. 
For an object $E\in \aA_{1/2}^p$ 
of numerical type $v$, it is 
 $\mu_{\sigma}$-limit (semi)stable if and only if the following 
conditions hold.

(a) For any pure one dimensional sheaf $G\neq 0$
which admits a strict epimorphism $E\twoheadrightarrow G$ in 
$\aA_{1/2}^p$, one 
has
\begin{align}\label{ineq1}
\mu_{\sigma}(G)>-\frac{3B\omega^2}{\omega^3}.
\quad \left(
\mbox{resp. }\mu_{\sigma}(G)\ge -\frac{3B\omega^2}{\omega^3}.
\right) \end{align}

(b) For any pure one dimensional sheaf $F\neq 0$
which admits a strict monomorphism $F\hookrightarrow E$ in 
$\aA_{1/2}^p$, one 
has 
\begin{align}\label{ineq2}
\mu_{\sigma}(F)<-\frac{3B\omega^2}{\omega^3}.
\quad \left(
\mbox{resp. }\mu_{\sigma}(F)\le -\frac{3B\omega^2}{\omega^3}.
\right)
\end{align}
\end{prop} 
\begin{proof}
By Lemma~\ref{same} and
applying the same argument of~\cite[Lemma~3.4]{Tolim}, 
an object $E\in \aA_{1/2}^{p}$ of numerical type $v$
is $Z_{\mu_{\sigma}}$-limit semistable if and only if 

(a') For any pure one dimensional sheaf $G\neq 0$
which admits an exact sequence  
 $0 \to F \to E\to G\to 0$ in 
$\aA_{1/2}^p$, one 
has $Z_{\mu_{\sigma}}(F) \preceq Z_{\mu_{\sigma}}(G)$. 

(b') For any pure one dimensional sheaf $F\neq 0$ which admits 
an exact sequence 
 $0 \to F \to E\to G \to 0$ in 
$\aA_{1/2}^p$, one 
has $Z_{\mu_{\sigma}}(F) \preceq Z_{\mu_{\sigma}}(G)$. 

By Lemma~\ref{lem:below} below, the
 conditions (a'), (b') are equivalent to (a), (b)
respectively. 
\end{proof}

\begin{lem}\label{lem:below}
Take $v_1, v_2 \in C(\aA_{1/2}^p)$
with 
$\ch(v_1)=(-1, 0, \beta_1, n_1)$, $\ch(v_2)=(0, 0, \beta_2, n_2)$, 
and $\beta_2 \neq 0$. 
Then $Z_{\mu_{\sigma}}(v_1)\preceq Z_{\mu_{\sigma}}(v_2)$, 
(\mbox{resp. }$Z_{\mu_{\sigma}}(v_1)\succeq Z_{\mu_{\sigma}}(v_2)$,)
if and only if 
\begin{align}\label{ineqmu}
\mu_{\sigma}(v_2) \ge -\frac{3B\omega ^2}{\omega^3}, \quad 
(\mbox{resp. }\mu_{\sigma}(v_2) \le -\frac{3B\omega ^2}{\omega^3}.)
\end{align}
If $\sigma =k\omega +i\omega$ for $k\in \mathbb{R}$, 
(\ref{ineqmu}) is equivalent to 
\begin{align}\label{ineqmu2}
k\ge -\frac{1}{2}\mu_{i\omega}(v_2), \quad 
(\mbox{resp. }k\le -\frac{1}{2}\mu_{i\omega}(v_2). )
\end{align}
\end{lem}
\begin{proof}
An easy computation shows, 
\begin{align}\notag
Z_{\sigma_m}^{\dag}(v_1)
&=\frac{1}{2}m^2\omega^2 B +\frac{1}{6}m^3 \omega^3 i, \\
\label{compute}
Z_{\sigma_m}^{\dag}(v_2)&=-n_2+B\beta_2 +m\omega \beta_2 i.
\end{align}
Since $\omega\beta_2 >0$, we have 
\begin{align*}
Z_{\mu_{\sigma}}(v_1)\preceq Z_{\mu_{\sigma}}(v_2) \quad &
\Leftrightarrow \quad -\frac{m^2 \omega^2 B/2}{m^3 \omega^3 /6}
\preceq \frac{\mu_{\sigma}(v_2)}{m} \\
& \Leftrightarrow \quad \mu_{\sigma}(v_2) \ge
-\frac{3\omega^2B}{\omega^3}.
\end{align*}
If $\sigma=k\omega+i\omega$ with
$k\in \mathbb{R}$, then 
$\mu_{\sigma}(v_2)=\mu_{i\omega}(v_2)-k$ and 
$-3B\omega^2/\omega^3=-3k$. Hence (\ref{ineqmu}) 
is equivalent to (\ref{ineqmu2}). 
\end{proof}

\subsection{Moduli theory of $\mu$-limit semistable objects}
In this subsection, we establish a moduli theory of 
$\mu$-limit semistable objects. 
In~\cite[Theorem~1.1]{Tolim}, 
a moduli theory of $\sigma$-limit 
stable objects is studied, 
and the resulting moduli space is an algebraic subspace of 
Inaba's algebraic space~\cite{Inaba}. 
Since we need a moduli theory not only for stable 
objects but also semistable objects, the resulting 
space should not be an algebraic space in general, 
but an Artin stack.
So instead of working with Inaba's algebraic space, 
we use Lieblich's algebraic stack 
of objects $E\in D^b(X)$, 
satisfying the condition, 
\begin{align}\label{negzero}
\Ext_X^{i}(E, E)=0, \quad \mbox{ for all }i<0,
\end{align}
which we denote by $\mathfrak{M}$. More precisely, the stack 
$\mathfrak{M}$ is defined by the 2-functor, 
$$\mathfrak{M} \colon (\Sch/\mathbb{C}) \lr (\mbox{groupoid}), $$
which takes a $\mathbb{C}$-scheme $S$ to the groupoid
$\mathfrak{M}(S)$, whose objects consist of relatively perfect 
object $\eE \in D^b(X\times S)$
such that $\eE_s$ satisfies (\ref{negzero}) for any 
closed point $s\in S$. 
Lieblich~\cite{LIE} shows the following. 
\begin{thm}\emph{(\bf{\cite{LIE}})}
The 2-functor $\mathfrak{M}$ is an Artin stack of locally finite type. 
\end{thm}
For $v\in C(\aA_{1/2}^p)$ as in (\ref{num}), we
consider a moduli problem of $\mu$-limit 
semistable objects of numerical type 
$v' \in C_{\le v}(\aA_{1/2}^p)$, 
where $C_{\le v}(\aA_{1/2}^p)$ is given in (\ref{Cle}). 
 First we show the following. 
\begin{lem}\label{lem:eff}
For any $v' \in C_{\le v}(\aA_{1/2}^p)$, we have one 
of the following. 
\begin{itemize}
\item There is $\beta' \in C(X)$ and $n'\in \mathbb{Z}$ such that 
$\ch(v')=(-1, 0, \beta', n')$.
\item We have $\ch(v')=(-1, 0, 0, 0)$. 
\item There is $\beta' \in C(X)$ and $n' \in \mathbb{Z}$
such that $\ch(v')=(0, 0, \beta', n')$. 
\end{itemize}
\end{lem}
\begin{proof}
For $v'\in C_{\le v}(\aA_{1/2}^{p})$, 
let $v''=v-v' \in C(\aA_{1/2}^p)$. 
Since $\ch_0(v)=-1$, we have 
$\ch_0(v')=0$ or $\ch_0(v')=-1$. 
Suppose that $\ch_0(v')=0$ and take $E\in \aA_{1/2}^p$ with 
$[E]=v'$. Then $\hH^{-1}(E)$ must be torsion, hence 
$\hH^{-1}(E)=0$ since $E\in \aA_{1/2}^{p}$.
 Therefore $E$ is a non-zero one dimensional sheaf, 
 thus $\ch(v')=(0, 0, \beta', n')$ for some 
 $\beta'\in C(X)$ and $n' \in \mathbb{Z}$.
  
 In the latter case, 
we have $\ch_0(v'')=0$ thus $\ch_0(v'')=(0, 0, \beta'', n'')$
for some $\beta''\in C(X)$ and $n'' \in \mathbb{Z}$. 
Therefore $\ch_0(v')=(-1, 0, \beta', n')$ for some 
$\beta' \in \overline{C}(X)$ and $n' \in \mathbb{Z}$. 
If $\beta'=0$, then $\hH^{-1}(E)$
is a line bundle 
 and $\hH^0(E)$ is a
zero dimensional sheaf, by~\cite[Lemma~3.2]{Tolim}. Thus 
$E$ is isomorphic to a direct sum of $\hH^{-1}(E)[1]$ and 
$\hH^0(E)$, which contradicts to $E\in \aA_{1/2}^p$
unless $n'=\dim \hH^0(E)=0$. 
\end{proof}
For $\sigma=B+i\omega \in A(X)_{\mathbb{C}}$
and $v' \in C_{\le v}(\aA_{1/2}^p)$, let 
us consider the following (abstract) stacks, 
$$\mathfrak{M}^{v'}(Z_{\mu_{\sigma}}) \subset \mathfrak{Obj}(\aA^p)
\subset \mathfrak{M}, $$
where $\mathfrak{Obj}(\aA^p)$ is the stack of 
objects $E\in \aA^p$, and $\mathfrak{M}^{v'}(Z_{\mu_{\sigma}})$ is 
the stack of $\mu_{\sigma}$-limit semistable objects 
$E\in \aA_{1/2}^p$ of numerical type $v'$. 
We have the following. 
\begin{prop}\label{prop:open}
The substacks $\mathfrak{Obj}(\aA^p)$ and 
$\mathfrak{M}^{v'}(Z_{\mu_{\sigma}})$ are open 
substacks of $\mathfrak{M}$, hence they are 
Artin stacks of locally finite type.
 Moreover $\mathfrak{M}^{v'}(Z_{\mu_{\sigma}})$ is 
of finite type. 
\end{prop}
\begin{proof}
The openness of $\mathfrak{Obj}(\aA^p)\subset
\mathfrak{M}$ follows from~\cite[Lemma~3.14]{Tolim}. 
Let us take $v' \in C_{\le v}(\aA_{1/2}^p)$. 
Suppose first that $\ch(v')=(0, 0, \beta', n')$
for $\beta' \in C(X)$ and $n' \in \mathbb{Z}$. 
Then any $\mu_{\sigma}$-limit semistable object 
of numerical type $v'$ is a $(B, \omega)$-twisted semistable 
sheaf. 
(cf.~Remark~\ref{rmk:twisted}.)
Then it is well-known that $\mathfrak{M}^{v'}(Z_{\mu_{\sigma}})$
is open in $\mathfrak{M}$, and it is of finite type. 
(cf.~\cite[Proposition~3.9]{ToBPS}.)

Next suppose that $\ch(v')=(-1, 0, \beta', n')$. 
In this case 
the claim for $\mathfrak{M}^{v'}(Z_{\mu_{\sigma}})$ follows 
from the straightforward adaptation of the argument 
of~\cite[Section~3]{Tolim}. 
In fact using Lemma~\ref{prop:char}, we can show the 
 boundedness of $\mu_{\sigma}$-limit semistable 
objects of numerical type $v'$,
 and destabilizing objects in a family of objects
in $\aA_{1/2}^p$, 
 along with the same arguments 
of~\cite[Proposition~3.13, Lemma~3.15]{Tolim}. 
Then the same proof of~\cite[Theorem~3.20]{Tolim}
works to show the 
 openness of $\mathfrak{M}^{v'}(Z_{\mu_{\sigma}})\subset
\mathfrak{M}$. 
The boundedness of relevant $\mu_{\sigma}$-limit 
semistable objects implies
 that $\mathfrak{M}^{v'}(Z_{\mu_{\sigma}})$
is of finite type. 
\end{proof}

\subsection{Stable pairs and $\mu$-limit semistable objects}
The notion of stable pairs and their counting invariants 
are introduced by Pandharipande and 
Thomas~\cite{PT} to interpret the reduced Donaldson-Thomas 
theory geometrically. In~\cite[Section~4]{Tolim}, 
the relationship between 
PT-invariants and counting invariants of 
limit stable objects are discussed. In this subsection we
state the similar result for 
$\mu$-limit semistable objects. 
Since the proofs are straightforward adaptation of
the arguments in~\cite[Section~4]{Tolim}, we again leave the readers to 
check the detail. 
First let us recall the definition of stable pairs. 
\begin{defi}\emph{
A \textit{stable pair} on a Calabi-Yau 3-fold $X$ is data $(F, s)$, 
where $F$ is a pure one dimensional sheaf on $X$, and $s$
is a morphism 
$$s\colon \mathcal{O}_X \lr F, $$
whose cokernel is a zero dimensional sheaf. }
\end{defi}
To simplify the notation, we also include the pair 
$(F=0, s=0)$ in the definition of stable pairs. 
For a stable pair $(F, s)$, we have the associated two
term complex, 
\begin{align}\label{term}
I^{\bullet}=(\oO_X \stackrel{s}{\lr} F) \in D^b(X), 
\end{align}
where $F$ is located in degree zero. 
Note that the object $I^{\bullet}$
satisfies 
\begin{align*}
& I^{\bullet} \in \aA_{1/2}^p, \quad \det I^{\bullet}=\oO_X, \\
& \ch(I^{\bullet})=(-1, 0, \beta, n),
\end{align*}
for $\beta=\ch_2(F)$, $n=\ch_3(F)$. 
By abuse of notation, we also call an object (\ref{term})
as a stable pair. 
In~\cite{PT}, the moduli space of stable pairs
 is constructed as a projective variety, and denoted by 
 $P_{n}(X, \beta)$,
$$P_{n}(X, \beta)\cneq\{ (F, s) \mid (F, s)\text{ is a stable pair with }
 (\ch_2(F), \ch_3(F))=(\beta, n)\}. $$ 
Let
$\mathfrak{Obj}_{0}(\aA^p)$ be the closed fiber at 
the point $[\oO_X] \in \Pic(X)$ of the following morphism, 
$$\det \colon 
\mathfrak{Obj}(\aA^p) \ni E \longmapsto \det E \in \Pic(X).$$
\begin{defi}\emph{\label{def:Lst}
For $\beta \in \overline{C}(X)$, $n\in \mathbb{Z}$, 
and $\sigma \in A(X)_{\mathbb{C}}$, 
define
$\mathfrak{L}_{n}^{\mu_{\sigma}}(X, \beta)$ to be 
\begin{align}\label{Lst}
\mathfrak{L}_{n}^{\mu_{\sigma}}(X, \beta)
\cneq \mathfrak{M}^{v}(Z_{\mu_{\sigma}}) \cap \mathfrak{Obj}_{0}(\aA^p),
\end{align}
where $v\in C(\aA_{1/2}^p)$ satisfies $\ch(E)=(-1, 0, \beta, n)$. 
}
\end{defi}
Note that 
$\mathfrak{L}_{n}^{\mu_{\sigma}}(X, \beta)$ is the moduli stack of
$\mu_{\sigma}$-limit semistable objects $E\in \aA^p$
with $\det E=\oO_X$ and 
$[E]=v$. 
We shall compare $\mathfrak{L}^{\mu_{\sigma}}_n(X, \beta)$
and $P_{n}(X, \beta)$, when $\sigma$ is written as 
$\sigma =k\omega +i\omega$ with $k\in \mathbb{R}$. 
For $\beta \in \overline{C}(X)$, we set 
\begin{align}\label{NE}
&\overline{C}_{\le \beta}(X)\cneq \{ \beta' \in \overline{C}(X) \mid 
\beta-\beta'\in \overline{C}(X)\}, \\
&\notag C_{\le \beta}(X)\cneq \overline{C}_{\le \beta}(X) \setminus \{0\}.
\end{align}
\begin{defi}
\label{defmu}\emph{
For $\beta \in \overline{C}(X)$, 
we define $m(\beta)$ as follows. If $\beta=0$, 
we set $m(\beta)=0$. Otherwise $m(\beta)$ is 
\begin{align*}
m(\beta) \cneq \mathrm{min} \{ \ch_3(\oO_C) \mid 
C\subset X \mbox{ \rm{satisfies} }\dim C=1, 
[C] \in \overline{C}_{\le \beta}(X)\}.
\end{align*}
It is well-known that 
$\overline{C}_{\le \beta}(X)$ is a finite set and $m(\beta)>-\infty$, 
whose proofs are seen in~\cite[Lemma~3.9, Lemma~3.10]{Tolim}.
Thus Definition~\ref{defmu} makes sense. 
For $\beta \in C(X)$
 and $n\in \mathbb{Z}$, we define $\mu_{n, \beta} \in \mathbb{Q}$
to be
 \begin{align}\label{mu}
 \mu_{n, \beta}\cneq \max \left\{ 
 \frac{n-m(\beta-\beta')}{\omega \beta'} : 
\beta' \in C_{\le \beta}(X) \right\}.
 \end{align}}
\end{defi}
The following
is $\mu$-stability version of~\cite[Theorem~4.7]{Tolim}.
\begin{thm}\label{prop:stack}
Let $\sigma=k\omega +i\omega$ for $k\in \mathbb{R}$. We have 
\begin{align}\label{comp1}
\mathfrak{L}^{\mu_{\sigma}}_{n}(X, \beta) 
&\cong [P_{n}(X, \beta)/\mathbb{G}_m], \quad \mbox{ if }k<-\mu_{n, \beta}/2, \\
\label{comp2}
\mathfrak{L}^{\mu_{\sigma}}_{n}(X, \beta) 
&\cong [P_{-n}(X, \beta)/\mathbb{G}_m], \quad \mbox{ if }k>\mu_{-n, \beta}/2.
\end{align}
Here $\mathbb{G}_m$ is acting on $P_{\pm n}(X, \beta)$ trivially. 
\end{thm}
\begin{proof}
The same proof of~\cite[Theorem~4.7]{Tolim}
shows that, if $k<-\mu_{n, \beta}/2$, then 
$E\in \aA_{1/2}^p$ is $\mu_{\sigma}$-limit semistable 
if and only if $E$ is isomorphic to a stable pair (\ref{term}). 
Note that~\cite[Lemma~4.6]{Tolim} is crucial 
in~\cite[Theorem~4.7]{Tolim}, and in our case
Proposition~\ref{prop:char} and 
(\ref{ineqmu2})
are applied instead of~\cite[Lemma~4.6]{Tolim}.  
Thus the $\mathbb{C}$-valued points of 
$\mathfrak{L}^{\mu_{\sigma}}_{n}(X, \beta)$ and $P_n(X, \beta)$ 
are identified. 

The existence of a universal stable pair on $X\times P_{n}(X, \beta)$
(cf.~\cite[Section~2]{PT})
yields a 1-morphism
$P_{n}(X, \beta) \to \lL_n^{\mu_{\sigma}}(X, \beta)$, which descends to
\begin{align}\label{yield}
[P_n(X, \beta)/\mathbb{G}_m]
 \lr \lL_n^{\mu_{\sigma}}(X, \beta).
\end{align}
Since any $E\in P_n(X, \beta)$ is $\tau$-limit 
stable for some $\tau$ by~\cite[Theorem~4.7]{Tolim},
 we have $\Hom(E, E)=\mathbb{C}$
and $\Aut(E)=\mathbb{G}_m$. 
Therefore (\ref{yield}) is an equivalence of groupoids on 
$\mathbb{C}$-valued points. 
As proved in~\cite[Theorem~2.7]{PT}, the 
stack $[P_n(X, \beta)/\mathbb{G}_m]$ is considered as 
an open substack of $\mathfrak{Obj}_0(\aA^p)$. 
 By Proposition~\ref{prop:open}, 
 $\lL_n^{\mu_{\sigma}}(X, \beta)$ is 
also open in $\mathfrak{Obj}_0(\aA^p)$,
hence
(\ref{yield}) 
gives an isomorphism of Artin stacks. 
The isomorphism (\ref{comp2}) is also similarly proved. 
\end{proof}

\section{Generating functions of stable pair invariants}
In this section, we combine the arguments in the previous sections to
show the rationality of the generating functions of 
stable pair invariants. 
As in the previous section, $X$ is a projective 
Calabi-Yau 3-fold, $\aA^p \subset D^b(X)$ is
the heart of a perverse t-structure on $D^b(X)$.

\subsection{Counting invariants of $\mu$-limit stable objects}
In this subsection, we construct counting invariants of 
$\mu$-limit semistable objects. 
Take $v\in C(\aA_{1/2}^p)$ which satisfies (\ref{num})
with $\beta \in \overline{C}(X)$ and $n\in \mathbb{Z}$. 
As in subsection~\ref{Count}, 
there is a Ringel-Hall Lie-algebra 
$\mathfrak{G}(\aA^p)\subset \hH(\aA^p)$
and the elements, 
$$\epsilon^{v'}(Z_{\mu_{\sigma}}) \in 
\mathfrak{G}^{v'}(\aA^p) \subset \hH(\aA^p), $$
for any $v' \in C_{\le v}(\aA_{1/2}^p)$
and $\sigma \in A(X)_{\mathbb{C}}$. Here we have used Theorem~\ref{prop:open}
which ensures the existence of 
$\hH(\aA^p)$ and 
$\epsilon^{v'}(Z_{\mu_{\sigma}})$. 
We also use the following map on $\hH(\aA^p)$ to construct 
the counting invariants, 
\begin{align}\label{Xi}
\Xi\colon \gG(\aA^p) \ni f \longmapsto 
f\cdot [\mathfrak{Obj}_{0}(\aA^p) \hookrightarrow 
\mathfrak{Obj}(\aA^p)] \in \gG(\aA^p).
\end{align}
The product $\cdot$ is given by (\ref{cdot}). 
In the following, we
 use the notation of subsection~\ref{Motivic}.
\begin{defi}\label{def:inv}
\emph{
For $\beta \in \overline{C}(X)$ and $n\in \mathbb{Z}$, 
we define
\footnote{The subscript $\ast ^{eu}$ means ``Euler characteristic '' of the moduli spaces.}
 $P_{n, \beta}^{eu}\in \mathbb{Z}$, 
$L_{n, \beta}^{eu}(\sigma)\in \mathbb{Q}$ and
$N_{n, \beta}^{eu}(\sigma)\in\mathbb{Q}$ 
to be 
\begin{align*}
P_{n, \beta}^{eu} &\cneq e(P_{n}(X, \beta)), \\
L_{n, \beta}^{eu}(\sigma) &\cneq
\Theta_v \Xi \epsilon^{v}(Z_{\mu_{\sigma}}), \quad 
\mbox{ where }\ch(v)=(-1, 0, \beta, n), \\
N_{n, \beta}^{eu}(\sigma) &\cneq 
\Theta_{v'} \epsilon^{v'}(Z_{\mu_{\sigma}}) =
J^{v'}(Z_{\mu_{\sigma}}), \quad 
\mbox{ where }\ch(v')=(0, 0, \beta, n).
\end{align*}
Here $e(\ast)$ is the topological Euler characteristic. 
}
\end{defi}
For simplicity we set 
$$L_{n, \beta}^{eu}\cneq L_{n, \beta}(i\omega), \quad 
N_{n, \beta}^{eu}\cneq N_{n, \beta}(i\omega).$$
\begin{rmk}\label{rmk:invLP}\emph{
Suppose that $\sigma=k\omega +i\omega$ 
for $k\in \mathbb{R}$. Then we have 
$$N_{n, \beta}^{eu}(\sigma)=N_{n, \beta}^{eu}, $$
by noting Remark~\ref{rmk:twisted}. Also if $k<-\mu_{n, \beta}/2$, 
then Theorem~\ref{prop:stack} and Remark~\ref{rmk:eu} 
imply 
$$L_{n, \beta}^{eu}(\sigma)=P_{n, \beta}^{eu}.$$
}
\end{rmk}

 Let us recall that for a fixed $\beta$, the moduli space 
 $P_{n}(X, \beta)$ is empty for a sufficiently negative $n$.
 (See~\cite{PT}.)
 So we can take $N(\beta) \in \mathbb{Z}$ such that 
 \begin{align}\label{Lau}
 P_{n, \beta}^{eu}=0 \quad \mbox{ for }n<N(\beta).
 \end{align}
 In particular the series $P_{\beta}^{eu}(q)$ is a
 Laurent polynomial of $q$. 

\subsection{Generating functions of counting invariants
 of $\mu$-limit semistable objects}
 In this subsection, we study the generating functions 
 of the invariants given in Definition~\ref{def:inv}. 
 Below we fix an ample divisor $\omega$ on $X$ and 
only consider the case $\sigma=k\omega +i\omega$ for 
$k\in \mathbb{R}$. 
For $\sigma=k\omega +i\omega$, we set $\sigma^{\vee}=-k\omega +i\omega$. 
We have the following symmetry for the invariants 
$L_{n, \beta}^{eu}(\sigma)$ and
$N_{n, \beta}^{eu}$. 
\begin{lem}\label{lem:sym}
(i) We have the equalities, 
\begin{align}\label{sym}
L_{n, \beta}^{eu}(\sigma)=L_{-n, \beta}^{eu}(\sigma^{\vee}), \quad 
N_{n, \beta}^{eu}=N_{-n, \beta}^{eu}.
\end{align}
(ii) For $d\cneq \omega \cdot \beta$, we have  
$N_{n+d, \beta}^{eu}=N_{n, \beta}^{eu}$ for any $n\in \mathbb{Z}$. 
\end{lem}
\begin{proof}
(i)
Let $\mathbb{D} \colon D^b(X) \to D^b(X)^{\rm{op}}$ be the 
dualizing functor, 
$$\mathbb{D}(E)=\dR \hH om(E, \oO_X[2]).$$
The functor $\mathbb{D}$ induces an isomorphism of 
rings, 
\begin{align}\label{bbD}
\mathbb{D}\colon \hH(\aA^p) \stackrel{\sim}{\lr}
 \hH(\mathbb{D}(\aA^p))^{\rm{op}}.
\end{align}
On the other hand, 
the functor $\mathbb{D}$ preserves the subcategory 
$\aA_{1/2}^p \subset D^b(X)$ by~\cite[Lemma~2.18]{Tolim}.
 Moreover 
the same argument of~\cite[Lemma~2.27]{Tolim} shows that
an object $E\in \aA_{1/2}^p$ is $\mu_{\sigma}$-limit 
semistable if and only if $\mathbb{D}(E) \in \aA_{1/2}^p$
is $\mu_{\sigma^{\vee}}$-limit semistable. 
Hence the map (\ref{bbD}) takes $\delta^{v'}(Z_{\mu_{\sigma}})$
to $\delta^{v^{'\vee}}(Z_{\mu_{\sigma^{\vee}}})$ for any
$v' \in C_{\le v}(\aA_{1/2}^p)$. Here 
if $v'$ is given by 
$\ch(v')=(r, 0, \beta', n')$ for $r=0$ or $-1$, then
 $v^{'\vee}$ is given by $\ch (v^{'\vee})=(r, 0, \beta', -n')$. 
Hence (\ref{sym}) follows by the definitions of $L_{n, \beta}^{eu}(\sigma)$
and $N_{n, \beta}^{eu}$.

(ii) Let us take an ample line bundle $\lL \in \Pic(X)$ 
with $c_1(\lL)=\omega$. 
The equivalence $\otimes \lL \colon \aA^p \to \aA^p$ 
induces an isomorphism of algebras, 
\begin{align}\label{tensor}
\otimes \lL \colon \hH(\aA^p) \lr \hH(\aA^p). 
\end{align}
On the other hand, it is easy to see that
an object $E\in \aA_{1/2}^p$ is $\mu_{\sigma}$-limit semistable 
if and only if $E\otimes \lL$ is $\mu_{\sigma}$-limit semistable.  
Thus the map (\ref{tensor}) takes $\delta^{v'}(Z_{\mu_{\sigma}})$
to $\delta^{v''}(Z_{\mu_{\sigma}})$, where 
$\ch(v')=(0, 0, \beta, n)$ and $\ch(v'')=(0, 0, \beta, n+d)$. 
Hence we can conclude $N_{n+d, \beta}^{eu}=N_{n, \beta}^{eu}$. 
\end{proof}
Next we show the following finiteness result. 
\begin{lem}\label{lem:finite}
For a fixed $\beta \in \overline{C}(X)$
 and $\sigma=k\omega +i\omega$, the set 
\begin{align}\label{set}
\{n\in \mathbb{Z} \mid L_{n, \beta}^{eu}(\sigma)\neq 0 \}
\end{align}
is a finite set. 
\end{lem}
\begin{proof}
If $\beta=0$, then $L_{n, \beta}^{eu}(\sigma)=0$
unless $n=0$ by Lemma~\ref{lem:eff}. 
Suppose that $\beta \neq 0$, and
take an integer $N(\beta)$ as in (\ref{Lau}). 
Assume that $L_{n, \beta}^{eu}(\sigma)\neq 0$
for $n<N(\beta)$. Then at least the moduli stack 
$\mathfrak{L}_{n}^{\sigma}(X, \beta)$ is non-empty, 
hence
we must have $k \ge -\frac{1}{2}\mu_{n, \beta}$
by Theorem~\ref{prop:stack}.
By the definition of $\mu_{n, \beta}$, there 
is $\beta' \in C_{\le \beta}(X)$
such that 
\begin{align*}
k\ge -\frac{n-m(\beta-\beta')}{2\omega \beta'}.
\end{align*}
Hence we have either
$$n\ge -2k(\omega \beta')+m(\beta-\beta'), \quad \mbox{or} \quad 
n\ge N(\beta).$$
Thus the set (\ref{set}) is bounded below. 
Since $L_{n, \beta}^{eu}(\sigma)=L_{-n, \beta}^{eu}(\sigma^{\vee})$
by Lemma~\ref{lem:sym}, 
the set (\ref{set}) is also bounded above. 
\end{proof}
\begin{lem}\label{cor:L}
For a fixed $\beta \in \overline{C}(X)$, the generating series 
\begin{align}\label{cor:se}
L_{\beta}^{eu}(q)=\sum_{n\in \mathbb{Z}}L_{n, \beta}^{eu}q^{n}
\end{align}
is a polynomial of $q^{\pm 1}$, hence a rational function of $q$, 
invariant under $q \leftrightarrow 1/q$. 
\end{lem}
\begin{proof}
By Lemma~\ref{lem:finite}, the series $L_{\beta}^{eu}(q)$ is a polynomial of 
$q^{\pm 1}$. Since 
$$L_{n, \beta}^{eu}=L_{n, \beta}^{eu}(i\omega)=
L_{-n, \beta}^{eu}(i\omega^{\vee})
=L_{-n, \beta}^{eu}(i\omega)=L_{-n, \beta}^{eu}, $$
by Lemma~\ref{lem:sym}, the polynomial $L_{\beta}^{eu}(q)$ is 
invariant under $q\leftrightarrow 1/q$. 
\end{proof}

\begin{lem}\label{lem:N}
The generating series 
$$N_{\beta}^{eu}(q)=\sum_{n\ge 0}nN_{n, \beta}^{eu}q^n, $$
is the Laurent expansion of a rational function of $q$, 
invariant under $q\leftrightarrow 1/q$. 
\end{lem}
\begin{proof}
Let $d\in \mathbb{Z}$ be as in Lemma~\ref{lem:sym}. Applying
 Lemma~\ref{lem:sym}, we have 
\begin{align*}
& N_{\beta}^{eu}(q)  \\
&= \sum_{j=0}^{d-1}\sum_{m\ge 0}
(dm+j)N_{j, \beta}^{eu}q^{dm+j} \\
&= \frac{1}{2}\sum_{j=1}^{d-1}\sum_{m\ge 0}\left\{
(dm+j)N_{j, \beta}^{eu}q^{dm+j}+
(dm+d-j)N_{d-j, \beta}^{eu}q^{dm+d-j}\right\} 
+\sum_{m\ge 0}dm N_{0, \beta}^{eu}q^{dm} \\
&= \sum_{j=1}^{d-1}\frac{N_{j, \beta}^{eu}}{2}\sum_{m \ge 0}\left\{
(dm+j)q^{dm+j}+(dm+d-j)q^{dm+d-j} \right\} 
+N_{0, \beta}^{eu}\sum_{m\ge 0}dm q^{dm}
\end{align*}
We can calculate as 
\begin{align}\notag
& \sum_{m \ge 0}\left\{
(dm+j)q^{dm+j}+(dm+d-j)q^{dm+d-j} \right\} \\
\label{cal1}
&= \frac{(d-j)(q^{j+d}+q^{d-j})+j(q^{j}+q^{2d-j})}{(1-q^d)^2},
\end{align}
and 
\begin{align}\label{cal2}
\sum_{m\ge 0}dm q^{dm} =\frac{dq^d}{(1-q^d)^2}.
\end{align}
Then the assertion follows since (\ref{cal1}), (\ref{cal2}) 
are rational functions of 
$q$, invariant under $q\leftrightarrow 1/q$.
\end{proof}
For a fixed $\beta \in \overline{C}(X)$, we consider the 
following generating series, 
$$P_{\beta}^{eu}(q)=\sum_{n\in \mathbb{Z}}P_{n, \beta}^{eu}q^n.$$
Now we state our main theorem in this paper. 
 \begin{thm}\label{thm:main2}
We have the following equality of the generating series, 
 \begin{align}\label{expan}
 \sum_{\beta}P_{\beta}^{eu}(q)v^{\beta}=
  \left( \sum_{\beta}L_{\beta}^{eu}(q) v^{\beta} \right) \cdot
 \exp \left(\sum_{\beta}N_{\beta}^{eu}(q)v^{\beta}\right). 
 \end{align}
 \end{thm}
 Combining Theorem~\ref{thm:main2}
 with Lemma~\ref{cor:L} and Lemma~\ref{lem:N}, we obtain the 
 following.
 \begin{cor}\label{cor:main}
 The generating series $P_{\beta}^{eu}(q)$ is the Laurent 
 expansion of a rational function of $q$, invariant under $q\leftrightarrow 
 1/q$. 
 \end{cor} 
 The proof of Theorem~\ref{thm:main2} will be given 
 in subsection~\ref{sub:fin} below.

\subsection{Transformation of the invariants 
$L_{n, \beta}^{eu}(\sigma)$}
In this subsection, we investigate the transformation 
formula of our invariants $L_{n, \beta}^{eu}(\sigma)$
under change of $\sigma=k\omega +i\omega$. 
For $\beta \in C(X)$, 
we set $S(\beta) \subset \mathbb{R}$ as
$$S(\beta)\cneq \left\{ \frac{m}{2\omega \beta'} \mid \beta '
\in C_{\le \beta}(X) , m\in \mathbb{Z}\right\}\subset \mathbb{R}.$$
Note that $S(\beta)$ is a discrete subset in $\mathbb{R}$ 
because $C_{\le \beta}(X)$ is a finite set. 
For $k_0 \in S(\beta)$, let 
 $\cC_{\pm}\subset \mathbb{R}\setminus 
S(\beta)$
be the connected components 
 such that $\cC_{-}\subset \mathbb{R}_{<k_0}$
and $\cC_{+}\subset \mathbb{R}_{>k_0}$. 
We take $k_{\pm}\in \cC_{\pm}$ and set 
$\sigma_{\ast}=k_{\ast}\omega +i\omega$ for $\ast=\pm, 0$.
We have the following. 
\begin{lem}\label{domi}
Take $v\in C(\aA_{1/2}^p)$ as in (\ref{num}) and 
$0 \in T=\mathbb{R}(m)$. 
The weak stability condition
$Z_{\mu_{\sigma_0}}$ dominates $Z_{\mu_{\sigma_{\pm}}}$
with respect to $(v, 0)$. 
\end{lem}
\begin{proof}
Take $v_1, v_2 \in C_{\le v}(\aA_{1/2}^p)$, and suppose that 
$Z_{\mu_{\sigma_{\pm}}}(v_1) \preceq Z_{\mu_{\sigma_{\pm}}}(v_2)$. 
We want to show that 
\begin{align}\label{want}
Z_{\mu_{\sigma_{0}}}(v_1) \preceq Z_{\mu_{\sigma_{0}}}(v_2).
\end{align}
By Lemma~\ref{lem:eff}, we have $\ch(v_i)=(r_i, 0, \beta_i, n_i)$ 
with $r_i=0$ or $-1$. If $r_1=r_2=-1$, 
it is easy to see that 
$Z_{\mu_{\sigma}}(v_1)=Z_{\mu_{\sigma}}(v_2)$ for any $\sigma$. 
If $r_1=r_2=0$, (\ref{want}) follows easily from (\ref{compute}). 
If $r_1 \neq r_2$, 
 (\ref{want}) follows from Lemma~\ref{lem:below}. 
 Then the assertion follows by noting
Remark~\ref{rmk:ob}.
\end{proof}
\begin{rmk}\emph{
Lemma~\ref{domi} is not true for the limit stability. 
This is the reason why we use $\mu$-limit stability 
rather than the limit stability.} 
\end{rmk}
For simplicity,
we fix $k\in \mathbb{R}_{<0}$ and write
 \begin{align}\label{simplicity}
 Z=Z_{\mu_{\sigma}} \mbox{ for }\sigma=k\omega +i\omega
\mbox{ with }k<0, \quad Z'=Z_{\mu_{i\omega}}.
\end{align}
Applying the results in the previous sections, we obtain the following
proposition. 
\begin{prop}
For $v\in C(\aA_{1/2}^p)$ as in (\ref{num}), and 
$Z, Z'$ as in (\ref{simplicity}) w.r.t. $(v, 0)$, 
 we have the following. 
\begin{align}\label{form2}
\epsilon^{v}(Z')=
\sum_{\begin{subarray}{c}l\ge 1, \ v_i \in C(\aA_{1/2}^p), \\
v_1+\cdots +v_l=v
\end{subarray}}U(\{v_1, \cdots, v_l\}, Z, Z')
\epsilon^{v_1}(Z)\ast \cdots \ast \epsilon^{v_l}(Z).
\end{align}
The sum (\ref{form2}) has only finitely many non-zero terms. 
\end{prop}
\begin{proof}
By Proposition~\ref{prop:open} and Remark~\ref{rmk:ob}, 
$Z, Z'$ satisfy
Assumption~\ref{assum2} w.r.t. $(v, 0)$.  
Furthermore by Lemma~\ref{domi}, the condition $(\spadesuit')$ 
is also satisfied with respect to $(v, 0)$, except the local finiteness 
condition which is satisfied if we knew that (\ref{form2}) 
has only finitely many non-zero terms. The finiteness of (\ref{form2})
will be shown in Proposition~\ref{Coe} (iii) below. 
 Therefore 
(\ref{form2}) follows from Theorem~\ref{thm:gen}. 
\end{proof}
In the formula (\ref{form2}), 
 $\{v_i\}_{i=1}^{l} \subset C(\aA_{1/2}^{p})$ satisfy the following
 by Lemma~\ref{lem:eff}. 
 \begin{align}\label{condition1}
&\mbox{there is  
a unique } 1\le e \le l \mbox{ such that }
\ch(v_i)=(0, 0, \beta_i, n_i) \mbox{ for }i\neq e \\
&\mbox{with }
\beta_i \in C(X), n_i \in \mathbb{Z}, 
\mbox{ and } \ch(v_e)=(-1, 0, \beta_e, n_e)\notag
\mbox{ with }\beta_e \in \overline{C}(X), n_e \in \mathbb{Z}.
\end{align}
Let us see that (\ref{form2}) is a finite sum.
For simplicity we write as 
$$\mu_i \cneq \mu_{i\omega}(v_i)=\frac{n_i}{\beta_i \omega}\in \mathbb{Q}, $$
if $v_i \in C(\aA_{1/2}^p)$ satisfies $\ch(v_i)=(0, 0, \beta_i, n_i)$. 
\begin{lem}\label{unless}
Take $v_1, \cdots, v_l \in C(\aA_{1/2}^{p})$ such that 
$\ch(v_i)=(0, 0, \beta_i, n_i)$ for all $i$ with 
$\beta_i \in C(X)$ and 
$n_i \in \mathbb{Z}$. 
For $Z$, $Z'$ as in (\ref{simplicity}), we have 
$$S(\{v_1, \cdots, v_l\}, Z, Z')= 
\left\{ \begin{array}{cc}1, & n=1, \\
0, & n\ge 2. \end{array} \right. $$
\end{lem}
\begin{proof}
Note that we have 
\begin{align}\label{Left}
Z(v_i) \preceq Z(v_j) \quad 
\Leftrightarrow \quad 
\mu_{i} \le \mu_{j}
\quad \Leftrightarrow \quad Z'(v_i)\preceq Z'(v_j).
\end{align}
Then one can check that 
the same proof of~\cite[Theorem~4.5]{Joy1} is applied. 
\end{proof}
\begin{lem}\label{unless2}
Take $v_1, \cdots, v_l \in C(\aA_{1/2}^{p})$ 
as in (\ref{condition1}). 
For $Z$, $Z'$ as in (\ref{simplicity}), 
$S(\{v_1, \cdots, v_l\}, Z, Z')$ is non-zero only if 
\begin{align}\label{muor}
0<\mu_{1}\le \mu_{2}\le \cdots \le \mu_{e-1} \le -2k 
>\mu_{e+1}>\mu_{e+2}>\cdots >\mu_{l} \ge 0.
\end{align}
\end{lem}
Moreover in this case, we have 
$S(\{v_1, \cdots, v_l\}, Z, Z')=(-1)^{e-1}$. 
\begin{proof}
Note that for $v_i, v_j$ with $i, j \neq e$, 
the same implications (\ref{Left}) hold. 
Also for $i\neq e$, we have 
\begin{align*}
Z(v_i)\preceq Z(v_e) \quad &\Leftrightarrow \quad 
\mu_{i} \le -2k, \\
Z'(v_i)\preceq Z'(v_e) \quad &\Leftrightarrow \quad \mu_{i}\le 0.
\end{align*}
Suppose that $S(\{v_1, \cdots, v_l\}, Z, Z') \neq 0$. 
We say $1< i< l$ is of type A (resp.~B) if the 
following holds, 
$$Z(v_{i-1})\succ Z(v_i) \preceq Z(v_{i+1}), \quad 
(\mbox{resp. }Z(v_{i-1})\preceq Z(v_i) \succ Z(v_{i+1}).)$$
If $1< i\le e-1$ is of type A, we have 
\begin{align*}
&Z'(v_1+\cdots +v_{i-1}) \preceq Z'(v_i +\cdots +v_l), \\
&Z'(v_1+\cdots +v_i) \succ Z'(v_{i+1}+\cdots +v_l). 
\end{align*}
Hence we have 
\begin{align*}
&\mu_{i\omega}(v_1+\cdots +v_{i-1})\le 0, \\ 
&\mu_{i\omega}(v_1+\cdots +v_{i-1}+v_i)>0, 
\end{align*}
which implies $\mu_i >0$. Similarly if $i$ is of type 
$B$, we have $\mu_i <0$. 

Suppose that there is $1\le i\le e-1$ of type A or B, and 
take the smallest such $i$. We assume $i$ is 
of type $A$, hence $\mu_i>0$.  
We have 
$$Z(v_1) \succ \cdots \succ Z(v_{i-1}) \succ Z(v_i) 
\preceq Z(v_{i+1}) \cdots,$$
thus $\mu_1>\cdots >\mu_i>0$ holds.
On the other hand, we have 
$Z'(v_1) \preceq Z'(v_2+\cdots +v_l)$, thus 
$\mu_1 \le 0$. This is a contradiction, so there is
no $1\le i\le e-1$ of type A. 
Similarly there is no $1\le i\le e-1$ of type $B$.  

By the above argument, one of (\ref{bycont}) or (\ref{bycont2}) holds. 
\begin{align}\label{bycont}
Z(v_1)\succ Z(v_2) \succ \cdots \succ Z(v_{e-1}) \succ Z(v_e), \\
\label{bycont2}
Z(v_1) \preceq Z(v_2) \preceq \cdots \preceq Z(v_{e-1}) \preceq Z(v_e).
\end{align}
Assume by a contradiction that (\ref{bycont}) holds.  
Then $Z'(v_1)\preceq Z'(v_2+\cdots +v_l)$, thus (\ref{bycont}) implies
\begin{align}\label{contmu}
0 \ge \mu_{1} >\mu_2> \cdots > \mu_{e-1}>-2k.
\end{align}
The inequality (\ref{contmu}) does not occur since we took $k<0$. 
Therefore we must have (\ref{bycont2}).  

A similar argument for $v_{e+1}, \cdots, v_l$ shows 
\begin{align}\label{iff1}
Z(v_1) \preceq \cdots \preceq Z(v_e) \succ Z(v_{e+1}) 
\succ \cdots \succ Z(v_l).
\end{align}
Obviously (\ref{iff1}) together with $S(\{v_1, \cdots, v_l\}, Z, Z')\neq 0$
imply (\ref{muor}).
\end{proof}

\begin{prop}\label{Coe}
(i)
Take $v_1, \cdots, v_l \in C(\aA_{1/2}^{p})$ 
as in (\ref{condition1}), and let $Z$, $Z'$ be as in (\ref{simplicity}). 
Then
$U(\{v_1, \cdots, v_l\}, Z, Z')$ is non-zero 
only if 
\begin{align}\label{onlyif}
0\le \mu_{1} \le \mu_{2} \le \cdots \le \mu_{e-1}\le -2k \ge \mu_{e+1} \ge 
\cdots \ge \mu_l \ge 0.
\end{align}
(ii) In the same situation of (i), suppose that 
$U(\{v_1, \cdots, v_l\}, Z, Z')\prod_{i\neq e}n_i$
is non-zero. Then we have \emph{
\begin{align}\notag
U(\{v_1, \cdots, v_l\}, Z, Z')=\sum_{1\le m\le l}
&\sum_{\begin{subarray}{c}
\text{surjective }\psi \colon \{1, \cdots, l\}
 \to \{1, \cdots, m\}, \\
 i\le j \text{ implies }\psi(i)\le \psi(j), \text{ and satisfies }
 (\ref{Usat}) \end{subarray} } \\
 \label{U3} &\qquad \qquad (-1)^{\psi(e)-1}\prod_{b=1}^{m}
\frac{1}{\lvert \psi^{-1}(b) \rvert !}.
\end{align}}
Here $\psi$ satisfies the following. \emph{
\begin{align}\label{Usat}
&\text{For } i, j <e \text{ with }\psi(i)=\psi(j), \text{ we have }
\mu_i=\mu_j, \text{ if } \psi(i)=\psi(e) \text{ then }
\mu_i=-2k, \\ 
\notag
& \text{and for } e<i, j, 
\text{ we have }\psi(i)=\psi(j) \text{ if and only if }
\mu_i=\mu_j.
\end{align}}
(iii) The formula (\ref{form2}) is a finite sum.  
\end{prop}
\begin{proof}
(i)
Suppose that $U(\{v_1, \cdots, v_l\}, Z, Z')\neq 0$ 
and take 
\begin{align}\label{phisi}
\psi \colon \{1, \cdots, l\} \to \{1, \cdots, m\}, \quad 
\xi\colon \{1, \cdots, m\} \to \{1, \cdots, m'\}, 
\end{align}
as in (\ref{U}). 
Let us set $c=\xi \psi(e) \in \{1, \cdots, m'\}$
and take $a \in \{1, \cdots, m'\}$ with $a\neq c$. 
By Lemma~\ref{unless}, 
the set $\xi^{-1}(a)$ consists of one element, 
say $b\in \{1, \cdots, m\}$. 
Then by the definition of $U(\{v_1, \cdots, v_l\}, Z, Z')$, 
we have 
$\mu_i=\mu_j$ for $i, j \in \psi^{-1}(b)$ and 
\begin{align}\label{latter}
Z'(\sum_{i\in \psi^{-1}(b)}v_i) = Z'(\sum_{i\in \psi^{-1}\xi^{-1}(c)}v_i).
\end{align}
The condition (\ref{latter})
 implies $\mu_{i\omega}(\sum_{i\in \psi^{-1}(b)}v_i)=0$, 
hence $\mu_i=0$ for any $i\in \psi^{-1}(b)$, i.e. $\mu_i=0$ for 
any $i\notin \psi^{-1}\xi^{-1}(c)$. 
By Lemma~\ref{unless2}, we must have (\ref{onlyif}). 

(ii) Suppose that
\begin{align}\label{supU}
U(\{v_1, \cdots, v_l\}, Z, Z')\prod_{i\neq e}n_i \neq 0, 
\end{align}
and take 
$\psi \colon \{1, \cdots, l\} \to \{1, \cdots, m\}$, 
$\xi\colon \{1, \cdots, m\} \to \{1, \cdots, m'\}$
as in (i).  
Then the proof of (i) shows that (\ref{supU}) 
is non-zero only if $m'=1$.
Then (\ref{U3}) follows from the definition of 
$U(\{v_1, \cdots, v_l\}, Z, Z')$
and Lemma~\ref{unless2}. 

(iii) Since $v_1, \cdots, v_l$ satisfy (\ref{condition1}), 
the number $l$ is bounded, and there is only finite number of 
possibilities for $\beta_i=\ch_2(v_i)$. Hence we may fix $l$ and 
$\beta_1, \cdots, \beta_l$. Then the values $n_i=\ch_3(v_i)$ have
 only finite number of possibilities by (\ref{onlyif}).  
\end{proof}
Now we have the wall-crossing formula of 
the invariants $L_{n, \beta}^{eu}(\sigma)$. 
\begin{prop}
For $\sigma=k\omega +i\omega$ with $k<0$, $\beta \in \overline{C}(X)$
and $n\in \mathbb{Z}$, we have the following formula, 
\emph{\begin{align}\notag
L_{n, \beta}^{eu}=&
\sum_{\begin{subarray}{c}l\ge 1, \ 1\le e\le l, \ 
\beta_i \in C(X) \text{ for }i\neq e, 
\ \beta_e \in \overline{C}(X), \ n_i \in \mathbb{Z},  \\
 \beta_1 +\cdots +\beta_l=\beta, \
n_1 +\cdots +n_l=n, \ \mu_i=n_i/\beta_i \omega \text{ satisfy }\\ 
 0<\mu_1 \le \mu_2 \le \cdots \le  \mu_{e-1} \le -2k \ge \mu_{e+1}
 \ge \cdots \ge \mu_l >0
\end{subarray}}
\sum_{\begin{subarray}{c}1\le m\le l, \ 
\text{surjective }\psi \colon \{1, \cdots, l\}
 \to \{1, \cdots, m\}, \\
 i\le j \text{ implies }\psi(i)\le \psi(j), \text{ and satisfies }
 (\ref{Usat})\end{subarray}} \\
&\qquad \label{Trans5.5}
\left( -\frac{1}{2} \right)^{l-1}
\prod_{b=1}^{m}\frac{(-1)^{\psi(e)-e}}{\lvert \psi^{-1}(b) \rvert !}
\prod_{i\neq e}n_i 
N_{n_i, \beta_i}^{eu}L_{n_e, \beta_e}^{eu}(\sigma).
\end{align}}
\end{prop}
\begin{proof}
Let $v\in C(\aA_{1/2}^p)$ be as in (\ref{num}), 
and $Z$, $Z'$ be as in (\ref{simplicity}). 
Applying $\Xi$ given in (\ref{Xi}) to (\ref{form2}), 
we obtain 
\begin{align}\notag
\Xi\epsilon^{v}(Z')=
\sum_{l\ge 1, \ 1\le e \le l}
\sum_{\begin{subarray}{c}
 v_i \in C(\aA_{1/2}^p), \
v_1+\cdots +v_l=v, \\
\ch_0(v_i)=0 \text{ for }i\neq e, \ \ch_0(v_e)=-1
\end{subarray}}&U(\{v_1, \cdots, v_l\}, Z, Z') \\
&\label{formula0.5}
\epsilon^{v_1}(Z)\ast \cdots \ast  \Xi\epsilon^{v_e}(Z)\ast 
 \cdots \ast \epsilon^{v_l}(Z).
\end{align}
Note that for $i\neq e$, the element
 $\epsilon^{v_i}(Z)\in \hH(\aA^p)$ is supported on 
$\mathfrak{Obj}_{0}(\aA^p)\subset \mathfrak{Obj}(\aA^p)$, 
hence $\Xi(\epsilon^{v_i}(Z) \ast \epsilon)=
\epsilon^{v_i}(Z)\ast \Xi(\epsilon)$
follows for any $\epsilon \in \hH(\aA^p)$. Thus 
(\ref{formula0.5}) follows from (\ref{form2}). 
Also we note that (\ref{formula0.5}) is a finite sum by 
Proposition~\ref{Coe} (iii). Hence applying $\Theta$ 
given in (\ref{Psi}) and using the same argument 
of Theorem~\ref{prop:trans}, 
we obtain 
\begin{align}
\notag
L_{n, \beta}^{eu}=
&\sum_{\begin{subarray}{c}l\ge 1, \ v_i \in C(\aA^p_{1/2}), \\
v_1+\cdots +v_l=v\end{subarray}}
\sum _{\begin{subarray}{c}
\Gamma \text{ \rm{is a connected, simply connected}} \\
\text{\rm{graph with vertex} }\{1, \cdots, l\}, \
\stackrel{i}{\bullet}\to \stackrel{j}{\bullet}\text{ \rm{implies} }
i< j
\end{subarray}} 
 \frac{1}{2^{l-1}}U(\{v_1, \cdots, v_l\}, Z, Z')  \\
&\qquad \label{Trans4}
\prod_{\stackrel{i}{\bullet} \to \stackrel{j}{\bullet}\text{ \rm{in} }\Gamma}
\chi(v_i, v_j)
\prod_{i\neq e}N_{n_i, \beta_i}^{eu}L_{n_e, \beta_e}^{eu}(\sigma),
\end{align}
where $v_1, \cdots, v_l$ satisfy (\ref{condition1})
and we have used the notation of (\ref{condition1}). 
By Riemann-Roch theorem, we have
 $\chi(v_i, v_j)=0$ for $i, j \neq e$, 
$\chi(v_i, v_e)=n_i$ and $\chi(v_e, v_i)=-n_i$. Hence 
a term in 
the sum (\ref{Trans4}) is non-zero only if the 
graph $\Gamma$ is of the following form, 
$$
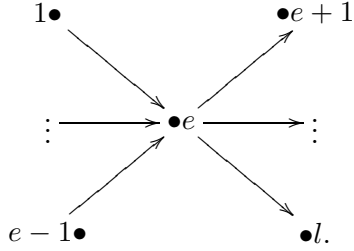
\xymatrix{
1 \bullet \ar[dr] &   & \bullet e+1 \\
\vdots \ar[r] & \bullet e \ar[ur] \ar[dr] \ar[r] &\vdots \\
e-1 \bullet \ar[ur] & & \bullet l.
}$$
Hence applying Proposition~\ref{Coe} (ii), we obtain the 
formula (\ref{Trans5.5}). 
\end{proof}

\subsection{Relationship between $L_{n, \beta}^{eu}$
and $P_{n, \beta}^{eu}$}
We next establish a relationship between 
$L_{n, \beta}^{eu}$ and $P_{n, \beta}^{eu}$. 
Let us take $N(\beta) \in \mathbb{Z}$ as in (\ref{Lau}). 
We choose $k<0$ so that 
\begin{align}\label{takek}
k<-\frac{1}{2}(n-N(\beta')), \quad 
k<-\frac{1}{2}\mu_{n, \beta'} \quad \text{ for any }
\beta' \in C_{\le \beta}(X).
\end{align} 
In this particular choice of $k$, we have the 
following formula. 
\begin{prop}\label{lem:Trans5}
If $k$ satisfies (\ref{takek}), then (\ref{Trans5.5}) implies 
the following. \emph{
\begin{align}\notag
L_{n, \beta}^{eu}&=
\sum_{\begin{subarray}{c}l\ge 1, \ 1\le e\le l, \ 0\le t \le e-1, \ 
0\le s \le l-e, \\
0=m_0<m_1< \cdots <m_t=e-1, \\
e=m_0'<m_1'<\cdots <m_s'=l \end{subarray}}
\sum_{\begin{subarray}{c} \beta_i \in C(X) \text{ for }i\neq e, \ 
\beta_e \in \overline{C}(X), \ n_i \in \mathbb{Z}, \\
\beta_1+\cdots +\beta_l=\beta, \
n_1+\cdots +n_l=n, \ \mu_i=n_i/\omega \beta_i \\
\text{satisfy }0<\mu_1=\cdots =\mu_{m_1}<\mu_{m_1+1}= \cdots, \\
0<\mu_{l}=\cdots 
=\mu_{m'_{s-1}+1}<\mu_{m'_{s-1}}=\cdots.
\end{subarray}} \\
\label{Trans6}
& \left( -\frac{1}{2}\right)^{l-1}
\prod_{i=1}^{t}\frac{1}{(m_{i}-m_{i-1})!}
\prod_{i=1}^{s}\frac{1}{(m_i'-m_{i-1}')!}\prod_{i\neq e}n_i
N_{n_i, \beta_i}^{eu}P_{n_e, \beta_e}^{eu}.
\end{align}}
\end{prop}
\begin{proof}
First note that all the $n_i$ in the formula (\ref{Trans5.5})
are positive except $i=e$. Thus
we have $n_e \le n$, hence 
$$k<-\mu_{n, \beta_e}/2 \le -\mu_{n_e, \beta_e}/2.$$
Therefore in the formula (\ref{Trans5.5}), we have 
$$L_{n_e, \beta_e}^{eu}(\sigma)=P_{n_e, \beta_e}^{eu}, $$
by Remark~\ref{rmk:invLP}. 
Thus we may assume $n_e \ge N(\beta_e)$ in
the formula (\ref{Trans5.5}). 
Then 
the condition $\mu_i\le -2k$ 
in (\ref{Trans5.5}) is automatically satisfied, since 
$$0<\mu_i \le n_i \le n-n_e \le n-N(\beta_e) <-2k, $$
by our choice of $k$. Hence we can eliminate the 
condition $\mu_i\le -2k$ in (\ref{Trans5.5}), 
and obtain the formula, 
\begin{align}\notag
L_{n, \beta}^{eu}=&
\sum_{\begin{subarray}{c}l\ge 1, \ 1\le e\le l, \ 
\beta_i \in C(X) \text{ for }i\neq e, 
\ \beta_e \in \overline{C}(X), \ n_i \in \mathbb{Z},  \\
 \beta_1 +\cdots +\beta_l=\beta, \
n_1 +\cdots +n_l=n, \ \mu_i=n_i/\beta_i \omega \text{ satisfy }\\ 
 0<\mu_1 \le \mu_2 \le \cdots \le  \mu_{e-1}, \ 
 0< \mu_l \le \mu_{l-1} \le \cdots \le \mu_{e+1}
\end{subarray}}
\sum_{\begin{subarray}{c}1\le m\le l, \ 
\text{surjective }\psi \colon \{1, \cdots, l\}
 \to \{1, \cdots, m\}, \\
 i\le j \text{ implies }\psi(i)\le \psi(j), \text{ and satisfies }
 (\ref{Usat})\end{subarray}} \\
&\qquad \label{Trans6.5}
\left(-\frac{1}{2}\right)^{l-1}
\prod_{b=1}^{m}\frac{(-1)^{\psi(e)-e}}{\lvert \psi^{-1}(b) \rvert !}
\prod_{i\neq e}n_i 
N_{n_i, \beta_i}^{eu}P_{n_e, \beta_e}^{eu}.
\end{align}
We
rearrange the sum (\ref{Trans6.5})
by first choosing partitions $0=m_0<m_1<\cdots <m_l=e-1$, 
$e=m_0'<m_1'<\cdots <m_s'=l$ and then choosing $\beta_i$, $n_i$
so that $0<\mu_{1}=\cdots =\mu_{m_1}<\mu_{m_1+1}=\cdots$
and $0<\mu_{l}=\cdots =\mu_{m_{s-1}'+1}<\mu_{m_{s-1}'}=\cdots$
are satisfied. Noting that $\psi$ satisfies (\ref{Usat}), we obtain 
\begin{align}\notag
&L_{n, \beta}^{eu}=
\sum_{\begin{subarray}{c}l\ge 1, \ 1\le e\le l, \ 0\le t \le e-1, \ 
0\le s \le l-e, \\
0=m_0<m_1< \cdots <m_t=e-1, \\
e=m_0'<m_1'<\cdots <m_s'=l \end{subarray}}
\sum_{\begin{subarray}{c} \beta_i \in C(X) \text{ for }i\neq e, \ 
\beta_e \in \overline{C}(X), \ n_i \in \mathbb{Z}, \\
\beta_1+\cdots +\beta_l=\beta, \
n_1+\cdots +n_l=n, \ \mu_i=n_i/\omega \beta_i \\
\text{satisfy }0<\mu_1=\cdots =\mu_{m_1}<\mu_{m_1+1}= \cdots, \\
0<\mu_{l}=\cdots 
=\mu_{m'_{s-1}+1}<\mu_{m'_{s-1}}=\cdots.
\end{subarray}}\left( -\frac{1}{2}\right)^{l-1} \\
\label{Trans6'}
& 
\prod_{i=1}^{t}\left(
\sum_{\begin{subarray}{c}l_i=m_i-m_{i-1}, \ 1\le m\le l_i, \\ 
\text{surjective } \psi \colon \{1, \cdots, l_i \}
 \to \{1, \cdots, m\}, \\
 i\le j \text{ implies }\psi(i)\le \psi(j)
 \end{subarray}} 
\prod_{b=1}^{m}\frac{(-1)^{l_i-m}}{\lvert \psi^{-1}(b)\rvert !} \right)
\prod_{i=1}^{s}\frac{1}{(m_i'-m_{i-1}')!}\prod_{i\neq e}n_i
N_{n_i, \beta_i}^{eu}P_{n_e, \beta_e}^{eu}.
\end{align}

Then (\ref{Trans6}) follows from Lemma~\ref{element} below.
\end{proof}
\begin{lem}\label{element}
For a fixed $l$, we have \emph{ 
\begin{align}
\sum_{\begin{subarray}{c} 1\le l\le m, \ 
\text{surjective }\psi \colon \{1, \cdots, l\}
 \to \{1, \cdots, m\}, \\
 i\le j \text{ implies }\psi(i)\le \psi(j)
 \end{subarray}}
 \prod_{b=1}^{m}\frac{(-1)^{l-m}}{\lvert \psi^{-1}(b)\rvert !}
=\frac{1}{l!}.
\end{align}}
\end{lem}
\begin{proof}
The proof is elementary, 
and this is a special case of~\cite[Proposition~4.9]{Joy4}. 
\end{proof}

\subsection{Proof of Theorem~\ref{thm:main2}}\label{sub:fin}
We finally give a proof of Theorem~\ref{thm:main2}. 
\begin{proof}
For a fixed data $l\ge 1$, $1\le e\le l$, $\beta_i\in C(X)$ ($i\neq e$)
 and 
$\beta_e \in \overline{C}(X)$, we set  
\begin{align}\notag
F_e(\beta_1, \cdots, \beta_l) &=\sum_{n\in \mathbb{Z}}
\sum_{\begin{subarray}{l} 0\le t\le e-1, \ 0\le s \le l-e, \\
0=m_0<m_1< \cdots <m_t=e-1, \\
e=m_0'<m_1'<\cdots <m_s'=l.
\end{subarray}}
\sum_{\begin{subarray}{c}n_i \in \mathbb{Z}, \ 
n_1+\cdots +n_l=n, \ \mu_i = n_i/\omega \beta_i \text{ satisfy} \\
0<\mu_1=\cdots =\mu_{m_1}<\mu_{m_1+1}= \cdots, \\
 0<\mu_{l}=\cdots
=\mu_{m'_{s-1}+1}<\mu_{m'_{s-1}}=\cdots.
\end{subarray}}\\
\label{Trans7}
& \prod_{i=1}^{t}\frac{1}{(m_{i}-m_{i-1})!}
\prod_{i=1}^{s}\frac{1}{(m_i'-m_{i-1}')!}\prod_{i\neq e}n_i
N_{n_i, \beta_i}^{eu}P_{n_e, \beta_e}^{eu}q^{n}.
\end{align}
By the formula (\ref{Trans6}), we have 
\begin{align}\notag
&\sum_{n, \beta}L_{n, \beta}q^n v^{\beta} \\
&= \notag
\sum_{l\ge 1, 1\le e \le l}\left(-\frac{1}{2} \right)^{l-1}
\sum_{\begin{subarray}{c} \beta_i \in C(X) \text{ for }i\neq e, \ 
\beta_e \in \overline{C}(X), \\
\beta_1+\cdots +\beta_l=\beta \end{subarray}}
F_e(\beta_1, \cdots, \beta_{e-1}, 
 \beta_e, \beta_{e+1}, \cdots,  \beta_l) \\
 \notag
 &= 
 \sum_{l\ge 1, 1\le e\le l}\sum_{\begin{subarray}{c}
 \kappa_1 \colon I_1 \to C(X), \ 
 \kappa_2 \colon I_2 \to C(X), \ 
 \beta_e \in \overline{C}(X), \\
 \sum_{i\in I_1}\kappa_1(i) +\sum_{i\in I_2}\kappa_2(i) +\beta_e =\beta
  \end{subarray}} 
 \frac{1}{(e-1)!(l-e)!} \left( -\frac{1}{2}\right)^{l-1} v^{\beta}\\
 &\label{lastsum}
 \sum_{\begin{subarray}{l} \lambda_1 \colon \{1, \cdots, e-1\} \stackrel{\sim}
 {\to} I_1, \\
  \lambda_2 \colon \{e+1, \cdots, l\} \stackrel{\sim}
 {\to} I_2 \end{subarray}}F_e(\kappa_1\lambda_1(1), \cdots, 
 \kappa_1\lambda_1(e-1), \beta_e, 
  \kappa_2\lambda_{2}(e+1), \cdots, 
 \kappa_2\lambda_{2}(l)).
\end{align}
Here 
$I_1$ and $I_2$ are finite sets with $\lvert I_1 \rvert =e-1$, 
$\lvert I_2 \rvert =l-e$.
Let us fix data 
$l\ge 1$, $1\le e\le l$, 
$\kappa_1 \colon I_1 \to C(X)$, $\kappa_2 \colon I_2 \to 
C(X)$ and $\beta_e \in \overline{C}(X)$, and consider the last sum 
of (\ref{lastsum}).  
If we also fix bijections 
$\lambda_1'\colon \{1, \cdots, e-1\} \to I_1$ and 
$\lambda_2' \colon \{e+1, \cdots, l\} \to I_2$, 
then the choices of $\lambda_1$, $\lambda_2$ 
in (\ref{lastsum})
correspond to 
the elements of the symmetric groups 
$\gamma \in \mathfrak{S}_{e-1}$, $\gamma' \in \mathfrak{S}_{l-e}$ 
respectively. 
Let us rewrite $\beta_i=\kappa_1 \lambda_1'(i)$ for $1\le i \le e-1$
and $\beta_i=\kappa_2 \lambda_2'(i)$ for $e+1 \le i \le l$. 
Then we have 
\begin{align}\notag
& \sum_{\begin{subarray}{l} \lambda_1 \colon \{1, \cdots, e-1\} \stackrel{\sim}
 {\to} I_1, \\
  \lambda_2 \colon \{e+1, \cdots, l\} \stackrel{\sim}
 {\to} I_2 \end{subarray}}
 F_e(\kappa_1\lambda_1(1), \cdots, 
 \kappa_1\lambda_1(e-1), \beta_e, 
  \kappa_2\lambda_{2}(e+1), \cdots, 
 \kappa_2\lambda_{2}(l))
 \\ 
  &= \sum_{\begin{subarray}{c} \gamma \in \mathfrak{S}_{e-1} \\
 \notag
 \gamma' \in \mathfrak{S}_{l-e} \end{subarray}}F_e(\beta_{\gamma(1)}, 
 \cdots, \beta_{\gamma(e-1)}, \beta_e, \beta_{\gamma'(e+1)}, 
 \cdots, \beta_{\gamma'(l)}) \\
 \label{Trans9} &=\sum_{\begin{subarray}{c}
 0\le t\le e-1, \ 0\le s\le l-e, \\
 0=m_0<m_1<\cdots <m_t=e-1 \\
 e=m_0'<m_1' \cdots <m_s'=l 
 \end{subarray}}
 \sum_{\begin{subarray}{c}
 \gamma \in \mathfrak{S}_{e-1} \\
 \gamma' \in \mathfrak{S}_{l-e}
 \end{subarray}}
 \prod_{i=1}^{t}\frac{1}{(m_i-m_{i-1})!} \prod_{i=1}^{s}
\frac{1}{(m_i'-m_{i-1}')!}G_{\gamma, \gamma'}, 
\end{align}
where $G_{\gamma, \gamma'}$ is given by
\begin{align}\notag
G_{\gamma, \gamma'}&=\sum_{n\in \mathbb{Z}} 
 \sum_{\begin{subarray}{c}n_i \in \mathbb{Z}, \
n_1+\cdots +n_l=n, \ \mu_i=n_i/\omega \beta_i \text{ satisfy} \\
0<\mu_{\gamma(1)}=\cdots =\mu_{\gamma(m_1)}<\mu_{\gamma(m_1+1)}= \cdots, \\
 0<\mu_{\gamma'(l)}=\cdots 
=\mu_{\gamma'(m'_{s-1}+1)}<\mu_{\gamma'(m'_{s-1})}=\cdots
\end{subarray}} 
 \prod_{i\neq e}n_i N_{n_i, \beta_i}^{eu}
 P_{n_e, \beta_e}^{eu}q^n.
 \end{align}
Note that for $\gamma^{\sharp}\in \prod_{i=1}^{t}\mathfrak{S}_{m_i-m_{i-1}}
\subset \mathfrak{S}_{e-1}$ 
and $\gamma^{\flat}\in \prod_{i=1}^{s}\mathfrak{S}_{m_i'-m_{i-1}'}
\subset\mathfrak{S}_{l-e}$, 
we have 
$$G_{\gamma, \gamma'}=G_{\gamma \gamma^{\sharp}, \gamma' \gamma^{\flat}}.$$
Since we have 
$$\prod_{i=1}^{t}
\lvert \mathfrak{S}_{m_i-m_{i-1}}\rvert
=\prod_{i=1}^{t}(m_i-m_{i-1})!, \quad
\prod_{i=1}^{s}\lvert \mathfrak{S}_{m_i'-m_{i-1}'}\rvert=
\prod_{i=1}^{s}(m_i'-m_{i-1}')!, $$
 (\ref{Trans9}) is written as 
 \begin{align} \label{writeG}
 (\ref{Trans9}) &=\sum_{\begin{subarray}{c}
 0\le t\le e-1, \ 0\le s\le l-e, \\
 0=m_0<m_1<\cdots <m_t=e-1, \\
 e=m_0'<m_1'< \cdots <m_s'=l 
 \end{subarray}}
 \sum_{\begin{subarray}{l} \gamma \in \mathfrak{S}_{e-1}, 
 \gamma(i)<\gamma(i') \text{ if }i, i'\in [m_{j}+1, m_{j+1}]
 \text{ for some }j \\
  \gamma' \in \mathfrak{S}_{l-e},
 \gamma'(i)<\gamma'(i') \text{ if }i, i'\in [m_{j}'+1, m_{j'+1}]
 \text{ for some }j 
 \end{subarray}} G_{\gamma, \gamma'} 
 \end{align}
 On the other hand, if we are given $n_i\in \mathbb{Z}_{>0}$ for $i\neq e$
 and $n_e \in \mathbb{Z}$ with $n_1+\cdots +n_l=n$, there 
 are unique $\gamma \in \mathfrak{G}_{e-1}$, $\gamma' \in \mathfrak{G}_{l-e}$
 and partitions $0=m_0<m_1< \cdots <m_t=e-1$, $e=m_0'<m_1'<\cdots <m_s'=l$
 such that 
 $\gamma(i)<\gamma(i')$ for $i, i' \in [m_{j}+1, m_{j+1}]$, 
 $\gamma'(i)<\gamma(i')$ for $i, i' \in [m_{j}'+1, m_{j+1}']$, and 
 $\mu_i=n_i/\omega \beta_i$ satisfy 
 \begin{align*}
 & 0<\mu_{\gamma(1)}=\cdots =\mu_{\gamma(m_1)}<\mu_{\gamma(m_1+1)}=\cdots, \\
 &  0<\mu_{\gamma'(l)}=\cdots 
=\mu_{\gamma'(m'_{s-1}+1)}<\mu_{\gamma'(m'_{s-1})}=\cdots.
\end{align*}
Therefore (\ref{writeG}) is written as 
\begin{align}\notag
 (\ref{writeG})&
 =\sum_{n\in \mathbb{Z}}\sum_{\begin{subarray}{c}n_1+\cdots +n_l=n, \\
 n_i \in \mathbb{Z}_{>0} \text{ for }i\neq e
 \end{subarray}}
 \prod_{i\neq e}n_iN_{n_i, \beta_i}^{eu}P_{n_e, \beta_e}^{eu}
 q^n \\
 \label{Trans10}
 &= \prod_{i\neq e}N_{\beta_i}^{eu}(q)\cdot P_{\beta_e}^{eu}(q).
 \end{align}
 Noting 
 $$\sum_{1\le e\le l} \frac{1}{(e-1)!(l-e)! 2^{l-1}}=\frac{1}{(l-1)!}, $$
 we obtain 
\begin{align}\label{logex}
(\ref{lastsum}) &= \sum_{l\ge 1} \sum_{\begin{subarray}{c}
\beta_i \in C(X) \text{ for }
i\neq l, \ \beta_l \in \overline{C}(X) \\
\beta_1 + \cdots +\beta_l=\beta
\end{subarray}}
\frac{(-1)^{l-1}}{(l-1)!}
\prod_{i\neq l}N_{\beta_i}^{eu}(q)\cdot P_{\beta_l}^{eu}(q)v^{\beta}.
\end{align}
The formula (\ref{logex}) implies (\ref{expan}) as desired. 
\end{proof}

\subsection{Problem of incorporating virtual classes to 
Joyce's work} 
Since invariants defined in Definition~\ref{def:inv}
are interpreted as Euler characteristics of moduli stacks, 
they are unlikely to be unchanged under deformations of $X$. 
In order to construct invariants which are unchanged under
deformations, we need to construct virtual moduli cycles
on the moduli spaces and integrate them. The resulting invariants are
Euler characteristics of the moduli spaces (up to sign) if the moduli 
spaces are non-singular, but in general 
they differ from euler numbers. Thus
in order to solve Conjecture~\ref{rat}, we have to
construct invariants involving virtual classes
and establish the formulas like (\ref{Trans2}). 
At this moment we are unable to overcome this problem. 
However if we could 
 involve virtual classes with 
Joyce's theory, then Conjecture~\ref{rat} for PT-theory 
follows along with the same argument in this paper.  
To state this, let us recall that the integrations of
virtual classes are also realized as weighted sums of 
certain constructible functions introduced by Behrend~\cite{Beh}. 
He shows that, for any scheme $M$, 
there is
a canonical constructible function $\chi_{M} \colon M \to \mathbb{Z}$
such that $\chi_{M}=(-1)^{\dim M}$ if $M$ is non-singular, and 
 if $M$ carries a symmetric perfect obstruction theory, 
we have 
$$\sharp^{\rm{vir}}M=\sum_{n\in \mathbb{Z}}ne(\chi_{M}^{-1}(n)).$$
Under the situation in this section, 
we shall address the following question.
\begin{footnote}
{The formulation of Problem~\ref{prob} is taught to the author
 by D.~Joyce. }
\end{footnote} 
\begin{problem}\label{prob}
Does there exist a map
$$\Theta' \colon \mathfrak{G}(\aA^p) \lr \mathfrak{g}(\aA^p), $$
such that 
the following conditions hold?
\begin{itemize}
\item 
For $v\in C(\aA^p)$,
suppose that $\mathfrak{M}^{v}(Z_{\mu_{\sigma}})$ is written 
as $[M/\mathbb{G}_m]$ for a scheme $M$. Then
$$\Theta'(\epsilon^{v}(Z_{\mu_{\sigma}}))=\sum_{n\in \mathbb{Z}}
n\Theta([[\chi_M^{-1}(n)/\mathbb{G}_m]
\hookrightarrow \mathfrak{Obj}(\aA^p)]).$$
\item 
For $v_1, v_2 \in C(\aA^p)$,
 we have 
\begin{align}\label{prob:nu}
[\Theta'(\epsilon^{v_1}(Z)), \Theta'(\epsilon^{v_2}(Z))]
=(-1)^{\chi(v_1, v_2)}\Theta'[\epsilon^{v_1}(Z), \epsilon^{v_2}(Z)].
\end{align}
\end{itemize}
\end{problem}
There should be sign change in (\ref{prob:nu}), because 
$\chi_M=(-1)^{\dim M}$
on a smooth variety $M$. 
We are unable to solve Problem~\ref{prob} at this moment, 
but
the techniques given in this paper yield
 the following. 
\begin{thm}
Suppose that Problem~\ref{prob} is true. Then Conjecture~\ref{rat}
is true for PT-theory. 
\end{thm}
\begin{proof}
It is enough to work over the invariants, defined by $\Theta'$. 
As a modification of Definition~\ref{def:inv}, 
 let us define $L_{n, \beta}(\sigma)$, 
$N_{n, \beta}(\sigma)$ to be
\begin{align*}
L_{n, \beta}(\sigma)&\cneq \Theta' \Xi \epsilon^v(Z_{\mu_{\sigma}}),
 \quad \text{ where }\ch(v) =(-1, 0, \beta, n), \\
N_{n, \beta}(\sigma)&\cneq \Theta' \epsilon^v(Z_{\mu_{\sigma}}),
 \quad \text{ where }\ch(v) =(0, 0, \beta, n).
 \end{align*}
 Then (\ref{prob:nu}) yields a similar wall-crossing formula 
 for $L_{n, \beta}(\sigma)$, and 
 $$L_{n, \beta}(\sigma)=(-1)^{\dim \Pic(X)}P_{n, \beta},$$
  for 
 $\sigma=k\omega+i\omega$ with $k <-\mu_{n, \beta}/2$. 
 Therefore the same proof of Theorem~\ref{thm:main2} works, and
 we have the similar expansion of the generating series $Z_{\rm{PT}}$
 as in (\ref{expan}). Then  
 Conjecture~\ref{rat} for $P_{\beta}(q)$ follows as a corollary.  
\end{proof}

Yukinobu Toda

Institute for the Physics and 
Mathematics of the Universe (IPMU), University of Tokyo, 

Kashiwano-ha 5-1-5, Kashiwa City, Chiba 277-8582, Japan

\textit{E-mail address}:toda@ms.u-tokyo.ac.jp, toda-914@pj9.so-net.ne.jp

\end{document}